\documentclass[a4paper,11pt,reqno,noindent]{amsart}
\usepackage{amsmath} 
\usepackage{amsfonts,amssymb,amsthm}
\usepackage{hyperref}

 \hypersetup{pdfmenubar=false, 
pdfnewwindow=true, 
colorlinks=false, 
linkcolor=blue, 
citecolor=blue, 
filecolor=magenta, 
urlcolor=cyan 
} 
\usepackage{graphicx}
\usepackage{subcaption}
\usepackage[english]{babel} \usepackage{newlfont} \usepackage{color}
\usepackage[body={15cm,21.5cm},centering]{geometry} \usepackage{fancyhdr}
 \usepackage{esint} \usepackage{enumerate}

\newcommand{\Rsym}{\R^{2\times 2}_{\mathrm{sym}}}
\usepackage[active]{srcltx}

\newtheorem{theorem}{Theorem} \newtheorem{lemma}{Lemma}
\newtheorem{proposition}{Proposition} 
 
\newtheorem{definition}{Definition}

\theoremstyle{definition} 
 
\newtheorem{remark}{Remark}

\newcommand{\secf}{\boldsymbol{I\!I}}

\newcommand{\sym}{\mathrm{sym}}

\newcommand{\dist}{\operatorname{dist}} \newcommand{\supp}{\operatorname{supp}}

\newcommand{\e}{\varepsilon}

\newcommand\ecke{\mathop{\hbox{\vrule height 7pt width .3pt depth 0pt \vrule
height .3pt width 5pt depth 0pt}}\nolimits}

\newcommand{\R}{\mathbb{R}} 
\newcommand{\N}{\mathbb{N}} 
 
\renewcommand{\d}{\mathrm{d}} \renewcommand{\L}{\mathbb{L}}
\newcommand{\A}{\mathcal{A}} \newcommand{\M}{\mathcal{M}}
 \newcommand{\id}{\mathrm{Id}}

\newcommand{\F}{\mathcal F} 
 
 \renewcommand{\H}{\mathcal{H}}
 
\renewcommand{\L}{{\mathcal L}}

\newcommand{\NN}{\mathbf{N}}

\newcommand{\gr}{\mathrm{gr}}
\newcommand{\Rnsym}{\R^{n\times n}_{\mathrm{sym}}}
\newcommand{\Rs}{R^{\mathrm{sym}}}
\newcommand{\gm}{\mathbf{g}}

\title{On a $\Gamma$-limit of Willmore functionals with additional curvature
  penalization term} \date{\today} 
\author[H. Olbermann] {Heiner
Olbermann} 

\begin{document}
\maketitle

\begin{abstract}
We consider the Willmore functional on graphs, with an
additional penalization of the area where the curvature is
non-zero. Interpreting the penalization parameter as a Lagrange multiplier, this
corresponds to the Willmore functional with a constraint on the area where
the graph is flat. Sending the
penalization parameter to $\infty$ and rescaling suitably, we derive the limit
functional in the sense of $\Gamma$-convergence. 
\end{abstract}

\section{Introduction}

\subsection{Motivation: A constrained Willmore problem}

The motivation for the present paper comes from  a constrained Willmore
problem. More precisely, let us  consider smooth hypersurfaces $M\subset\R^3$  of a fixed topology,  with 
constraints on the
amount of surface area where the surface is flat and non-flat
respectively. Here, by flat we mean that  the second fundamental form vanishes. Within this class, we are
interested in the variational problem
\begin{equation}
\inf_M\int_M \left(H^2-2 K\right) \d \H^2=\inf_M\int_M(\kappa_1^2+\kappa_2^2)\d\H^2\,,\label{eq:28}
\end{equation}
where $\kappa_1$,
$\kappa_2$ denote  the principal curvatures, $H=\kappa_1+\kappa_2$  the mean curvature, $K=\kappa_1\kappa_2$ the Gauss curvature,  and $\H^2$ is the
two-dimensional Hausdorff measure.

\medskip

Here we are going to simplify this problem in two ways: First, we are not going
to consider arbitrary surfaces, but only  graphs.  Secondly, we are going to
replace the constraint of having a fixed amount of non-flat surface area by a
penalization of the non-flat part. This is the usual attempt of capturing
constraints via the introduction of Lagrange multipliers. We will however not be   able to prove a rigorous equivalence between the constrained variational
  problem and the problem involving Lagrange multipliers.

The latter consists in, for  $\lambda>0$ and a set of graphs $M$   with fixed
surface area, in  the variational problems
\begin{equation}
\inf_M\int_M \left(H^2-2 K\right) \d \H^2+\lambda \H^2\left(\{x\in M:S_M\neq 0\}\right)\,,\label{eq:29}
\end{equation}
where $S_M$ denotes the second fundamental form of $M$. Additionally, boundary conditions or other constraints on $M$ may be imposed.

\medskip

Obviously, the shape of  minimizers  for such
a problem depend on the  penalization parameter $\lambda$. One expects that the
concentration of curvature increases with  $\lambda$,  i.e.,  the area where the
surface is flat
 becomes larger as $\lambda$ increases (for configurations of low energy). The main purpose of the
present paper is a rigorous investigation of the limit $\lambda\to \infty$ for
the variational problem \eqref{eq:29}.

\subsection{Statement of main result}
For any Borel set $U\subset\R^n$, 
let
$\M(U)$ denote the set of signed Radon measures on $U$.   We denote
by $\M(U;\R^p)$ the $\R^p$ valued  Radon measures on $U$. Let $\R^{n\times n}_{\mathrm{sym}}$ denote the symmetric $n\times n$
matrices, and let
$\M(U;\R^{n\times n}_\sym)$ denote the space of measures with values in the symmetric matrices, i.e., $\{\mu\in\M(U;\R^{n\times
  n}):\mu_{ij}=\mu_{ji}\text{ for }i\neq j\}$. 
For $\mu\in \M(U;\R^{p})$, let $|\mu|$ denote the total variation
measure. 
For $\mu\in \M(U;\R^p)$, we have by the Radon-Nikodym differentiation Theorem (see
Theorem \ref{thm:RN} below) that for $|\mu|$-almost every $x\in U$, the derivative
$\d\mu/\d|\mu|$ exists.
For any one-homogeneous function $h:\R^{p}\to\R$ and any $\mu\in
\M(U;\R^{p})$, we may hence define 
\[
h(\mu)=h\left(\frac{\d\mu}{\d|\mu|}\right)\d|\mu|\,.
\]
This is a well defined Borel measure.

\medskip

For $\xi\in\Rsym$, let $\tau_1(\xi), \tau_2(\xi)$ denote the eigenvalues
of $\xi$. 
We set
\[
\rho^0(\xi):=\sum_{i=1}^2 |\tau_i(\xi)|\,.
\]
We will repeatedly use the following estimates:
\begin{equation}
|\xi|\leq \rho^0(\xi)\leq 2|\xi|\,.\label{eq:30}
\end{equation}
Note that $\rho^0:\Rsym\to\R$ is sublinear and positively one-homogeneous.

\medskip

Let $\Omega\subset\R^2$ be an open  bounded  set
with smooth boundary,
 and let $u\in BH(\Omega)$; that is the space of $u\in
W^{1,1}(\Omega)$ such that $\nabla u\in BV(\Omega;\R^2)$, where $\nabla u$
denotes the approximate gradient of $u$. 
We will use the usual notation for the $BV$ function $\nabla u$:
$J_{\nabla u}$ denotes the jump set of $\nabla u$. On $J_{\nabla u}$, there
exists a measurable function $\nu_{\nabla u}$ with values in $S^2$ such that $\nabla u$ has
well defined limits $\nabla u^\pm$ on both sides of the hyperplane defined by
$\nu_{\nabla u}$. The set $S_{\nabla u}$ is the singular set of $D\nabla u$, i.e., the set where
$D\nabla u$ is not absolutely continuous w.r.t. $\L^2$. (The operator ``$D$''
denotes the distributional derivative.)  Furthermore, $C_{\nabla
  u}:=S_{\nabla u}\setminus J_{\nabla u}$. We have the decomposition
\[
D\nabla u=\nabla^2 u \L^2+(\nabla u^+-\nabla u^-)\otimes\nu_{\nabla u} \ecke J_{\nabla u}+
D^s\nabla u \ecke C_{\nabla
  u}\,.
\]
Let $C_0(\Omega)$ denote the completion of $C_c^0(\Omega)$ with respect to the $\sup$-norm. We say that a sequence $\mu_j\in\mathcal M(\Omega)$ converges weakly * in the sense of measures to $\mu\in \mathcal M(\Omega)$ if for any $f\in C_0(\Omega)$, we have
\[
  \int_\Omega f\d\mu_j\to \int_{\Omega}f\d\mu\,.
\]
The convergence of vector-valued measures is defined analogously.
For a sequence $u_j\in BV(\Omega)$, we say that $u_j\to u$ weakly * in $BV$ if
$u_j\to u$ in $W^{1,1}$ and $D u_j\to Du$ weakly * in the sense of measures.

For $v\in\R^2$ and
$\xi\in\R^{2\times 2}$, we define
\[
\begin{split}
  \NN(v)&=\frac{1}{\sqrt{1+|v|^2}}\left(\begin{array}{c}v\\-1\end{array}\right)\,,\\
  \gm(v)&=\id_{2\times 2}+v\otimes v\,,\\
  \secf(v,\xi)&=\frac{1}{\sqrt{1+|v|^2}}\xi\\
  S(v,\xi)&=\gm(v)^{-1/2}\secf(v,\xi)\gm(v)^{-1/2}\,.
\end{split}
\]
By this definition, $S(\nabla u(x),\nabla^2 u(x))\in\Rsym$ is the second fundamental form (or
shape operator) of the graph of $u$ at $u(x)$ in matrix form (supposing $u$ is
sufficiently smooth at $x$); its eigenvalues are the
principal curvatures of the graph.
Let $F_\lambda:\R^{2\times 2}\to \R$ be defined by
\[
F_\lambda(\xi)=\begin{cases} 0 &\text{ if }\xi=0\\
|\xi|^2+\lambda &\text{ else.}\end{cases}
\]
We define $\F_\lambda: W^{2,2}(\Omega)\to[0,+\infty)$ by
\begin{equation}
  \F_\lambda(u)=\lambda^{-1/2}\int_{\Omega}F_\lambda(S(\nabla u,\nabla^2
  u))\sqrt{1+|\nabla u|^2}\d x\,.\label{eq:31}
  \end{equation}
Note that up to a normalizing factor, for smooth functions $u$, the right hand side is precisely the functional introduced in the
previous subsection,
\[
\lambda^{1/2}\F_\lambda(u)=\int_{\gr(u)}\left(H^2-2 K\right)\d\H^2+\lambda
  \H^2\left(\{x\in\gr(u):S_{\gr(u)}\neq 0\}\right)\,\,,
\]
where $\gr(u)$ denotes the graph of $u$.
For $u\in W^{2,2}(\Omega)$, the right hand side in \eqref{eq:31} is 
 finite, since the Willmore integrand
$\sqrt{1+|\nabla u|^2}|S|^2$ is bounded from above by $|\nabla^2 u|^2$ (see
Lemma \ref{lem:basicSestimate} below).
Let $\arccos:[-1,1]\to[0,\pi]$ be the inverse function of
$\cos:[0,\pi]\to[-1,1]$, and for $v=(v_1,v_2)^T\in \R^2$, let $v^\bot=(-v_2,v_1)^T$.  We define $\F: BH(\Omega)\to[0,+\infty)$ by
\[
\begin{split}
  \F(u)&=2\int_\Omega \rho^0(S(\nabla u,\nabla^2 u))\sqrt{1+|\nabla u|^2}\d x\\
  &\quad+2\int_{C_{\nabla u}} \rho^0\left(S\left(\nabla
      u,\frac{\d D\nabla u}{\d |D\nabla u|}\right)\right)\sqrt{1+|\nabla u|^2}\d
  |D\nabla u|\ecke C_{\nabla u}\\
&\quad+2\int_{J_{\nabla u}} \arccos \NN (\nabla u^+)\cdot \NN (\nabla
    u^-)\sqrt{1+|\nu_{\nabla u}^\bot\cdot\nabla u|^2}\d\H^1\,.
\end{split}
\]
Again, the right hand side always exists and is finite, since $\rho^0(S(v,\xi))\sqrt{1+|v|^2}\leq
2|\xi|$, and hence the integrands can be estimated by the Lebesgue regular, jump
and Cantor part of the measure $\rho(D\nabla u)$ respectively.
Finally, let us write $\A=BH(\Omega)\cap W^{1,\infty}(\Omega)$.

\medskip

Our main result is the following theorem, which establishes the
$\Gamma$-convergence $\F_\lambda\to \F$ in the weak-* topology of $BH(\Omega)$.

\begin{theorem}
\label{thm:main}
  \begin{itemize}
  \item[(i)] Let $u_\lambda$ be a sequence in $W^{2,2}(\Omega)$ with
    $\limsup_{\lambda\to\infty}\F_\lambda(u_\lambda)<\infty$, $\int_\Omega
    u_\lambda\d x=0$ and $\|\nabla u_\lambda\|_{L^\infty}<C$. Then there exists a
    subsequence (no relabeling) and $u\in\A$ such that

    \begin{equation}
u_\lambda\to u  \text{ in }W^{1,1}(\Omega), \quad\nabla u_\lambda\to \nabla
u\text{ weakly * in }BV(\Omega;\R^2)\,.\label{eq:22}
\end{equation}
\item[(ii)]
Let $u_\lambda$, $u$ be as in \eqref{eq:22}. Then we have
\[
\liminf_{\lambda\to\infty} \F_{\lambda}(u_{\lambda})\geq \F(u)\,.
\]
\item[(iii)]
Let $u\in \A$. Then there exists a sequence $u_\lambda$ such that \eqref{eq:22}
is fulfilled and 
\[
\limsup_{\lambda\to\infty}\F_\lambda(u_\lambda)\leq \F(u)\,.
\]
  \end{itemize}
\end{theorem}

\begin{remark}
  \begin{itemize}
  \item[(i)] For $u\in C^2(\Omega)$, the limit functional $\F$ can be written as
\begin{equation}
\label{eq:41}
\F(u)=\int_{\gr(u)}2\rho^0(S_{\gr(u)})\d\H^2\,,
\end{equation}
where  $\gr(u)$ denotes the graph of $u$. 
The formula for $\F$ from the statement of the theorem is a generalization
for surfaces whose second fundamental form is a measure. We note that graphs of
functions in $BH(\Omega)$ do not belong to the class of curvature varifolds as
defined in  \cite{hutchinson1986second,mantegazza1998curvature}. The latter do
not allow for a Cantor part in the curvature measure. 
\item[(ii)] For the ``geometrically linearized'' functionals
\[
\mathcal{G}_\lambda(u)=\lambda^{-1/2}\int_\Omega F_\lambda(\nabla^2 u)\d
x
\]
we have shown in \cite{olbermann2017michell} that the limit functional (again in
the sense of $\Gamma$-convergence) is given by $\mathcal
G(u)=2\int_\Omega\d\left(\rho^0(D\nabla u)\right)$. Here we merely replace the
second derivative $\nabla^2 u$ by the second fundamental form $S(\nabla
u,\nabla^2 u)$. However, the presence of lower order terms makes the analysis
more difficult for several reasons. There exist a few different techniques for the
proof of lower semicontinuity of integral functionals that depend on lower order
terms starting from the results without those terms, see
\cite{marcellini1985approximation,acerbi1984semicontinuity,fonseca1998analysis}. These
techniques do not work here since we consider  the convergence $\nabla
u_\lambda\to \nabla u$ weakly * in $BV$ (and not in $W^{1,p}$ with $p>1$ as in
the quoted references). The lower semicontinuity in $BV$ for integral functionals that
depend on lower order terms has been treated in \cite{MR1218685}. Their
technique cannot be applied in a straightforward way here either, the reason
being that for fixed $\lambda$ the integrands of our functionals have 2-growth
at infinity. Our technique will be a modification of the one from
\cite{MR1218685}, choosing a cutoff that despite the 2-growth does not increase
the energy by too much. 

Carrying on with the comparison of our result with the one in \cite{MR1218685},
we would like to point out that we are able to determine the form of the $\Gamma$-limit
on the jump part explicitly, which is not possible in the general situation
treated in \cite{MR1218685}. This requires the solution of a certain variational
problem that we obtain through some geometric considerations (see Section
\ref{sec:geometry}).

Concerning the upper bound, this is more difficult here than in
\cite{olbermann2017michell} again because of the presence of  lower order
terms. In that reference, the upper bound follows directly from well known
properties of approximations of $BV$ functions by mollification. Here, we need
to keep track of the behavior of the lower order terms in this approximation
process, for which we need to use some results on the fine properties of $BV$ functions.

\item[(iii)]
The requirement $\|\nabla u_\lambda\|_{L^\infty}<C$ in the
  compactness part of the theorem (statement (i)) may seem unnatural. Without such an
  assumption however, we are not able to obtain control of the $BH$-norm from
  the energy alone. This can be seen by considering graphs of functions with
  almost vertical parts. The energy of these almost vertical parts can be made
  arbitrarily small.  In this way, we might obtain functions of arbitrarily
  large $L^1$ norm with bounded energy. This can be considered as an artefact of
  the restriction to graphs, and shows that a geometric description would be
  more appropriate. This will be the topic of future work.

The requirement $\int_\Omega u_\lambda \d x=0$ is included in the statement (i) to enforce the
convergence $u_\lambda\to u$ in $L^1$. Without such an assumption, we would
still have the convergence $\nabla u_\lambda\to \nabla u$ weakly * in $BV$ (for
a subsequence).
\end{itemize}

\end{remark}

\subsection{Scientific context}
Vesicles of polyhedral shape play an important role in biology. Examples are
virus capsids \cite{caspar1962physical,lidmar2003virus}, carboxysomes \cite{yeates2008protein}, cationic-ionic vesicles, and assembled
supramolecular structures \cite{macgillivray1999structural}. 
In \cite{PNAS}, a model
for the formation of polyhedral structures 
based on minimization of the free elastic  energy of topologically spherical
shells has been suggested. In the model, the free energy
is a function of the deformation of the shell, and the material distribution of
the two elastic components that the shell is made of. 

Elastic inhomogeneities are known to exist in many virus
capsids and for carboxysomes; in both of these cases, the vesicle shell is made
up of different protein types. In \cite{PhysRevE.85.050501},
 it has been suggested that the inhomogeneities can  act as the
driving force for faceting.
In this reference, it is assumed that the vesicle wall consists of two
components, with different elastic properties (``soft'' and ``hard''), and the amount of soft and hard material available for the formation of the vesicle is fixed.
The variational problems \eqref{eq:28} and \eqref{eq:29},
interpreted as minimization problems for the free elastic energy, are models for
such two-phase vesicles. Following this interpretation, we investigate here the
limit in which the contrast between soft and hard phase is large (the hard
phase does not bend at all), and there is
a very small amount of soft material.

\subsection{Comparison to the analogous one-dimensional problem}

Consider the following variational problem, which is a lower dimensional analogue for problem
\eqref{eq:29}, with the topology fixed to be that of a
sphere instead of a graph:
\[
\begin{split}
  \inf \left\{\int_{M} \kappa^2 \d s+\lambda\H^1(\{x\in M:\kappa\neq 0\}):\text{ $M$ homeomorphic to $S^1$},\,
 \H^1(M)=2\pi\right\}
\end{split}
\]
Pulling the penalization term into the integral, we obtain
\[
\begin{split}
  \inf \left\{\int_{M} \tilde f_\lambda(\kappa) \d s :\text{ $M$ homeomorphic to $S^1$},\, \H^1(M)=2\pi\right\}\,,
\end{split}
\]
with
\[
\tilde f_\lambda(\kappa)=\begin{cases}\lambda+\kappa^2&\text{ if } \kappa\neq 0\\
0&\text{ else.}\end{cases}
\]
It is well known that such a problem requires relaxation to guarantee the
existence of minimizers. The relaxed problem is obtained by  replacing the integrand
with its convex lower semicontinuous envelope, 
\[
\tilde f_\lambda^{**}(\kappa)=\begin{cases}2\sqrt{\lambda}|\kappa|&\text{ if }|\kappa|\leq\sqrt{\lambda}\\
\lambda+\kappa^2&\text{ else.}
\end{cases}
\]
We see immediately  that the integrands $\lambda^{-1/2}\tilde f_\lambda^{**}$
are monotone decreasing, and converge to the function $\kappa\mapsto
2|\kappa|$. From this convergence, one  deduces without difficulty the
$\Gamma$-convergence of the respective integral functionals, with respect to weak *
convergence of the curvatures. The limit
functional $\F^{(1)}:M\mapsto 2\int_M|\kappa|\d s$ is also defined for  one-spheres whose
curvature is only a measure. Note that there is a large set of minimizers for
$\F^{(1)}$: Any one-sphere with non-negative curvature will be a
minimizer. 

\medskip

The situation in dimension two is completely different: From Theorem
\ref{thm:main}, it is natural to conjecture that one may define a limit
functional in the sense of $\Gamma$ convergence that for smooth surfaces is
given by \eqref{eq:41}.
For surfaces
of convex bodies, this
functional  is the same as the
total mean curvature. For sufficiently smooth surfaces, it is
known that the only minimizer of this functional within the class of topological
two-spheres is the round sphere, see
\cite{minkowski1989volumen,bonnesen1926quelques}. 

\subsection{Some notation, plan of the paper}
The
symbol ``$C$'' is used as follows: A statement such as ``$f\leq Cg$'' is shorthand for
``there exists a constant $C>0$  such that $f\leq
Cg$''. The value of $C$ may change within the same line. 
For $f\leq Cg$, we also write $f\lesssim g$.

 For $\xi\in \Rsym$, 
let $\tau_i(\xi)$, $i=1,2$ denote the eigenvalues of  $\xi$. We
denote the operator norm of $\xi$ by 
\[
|\xi|_{\infty}=\max(|\tau_1(\xi)|,|\tau_2(\xi)|)\,.
\]
The two-dimensional Lebesgue measure is denoted by $\L^2$, the $d$-dimensional
Hausdorff measure by $\H^d$. For $x=(x_1,x_2)^T\in\R^2$, we write
$x^\bot=(-x_2,x_1)^T$.

The following objects will be defined in $\R^n$, but most often we are going to consider  the case $n=2$. It will be obvious from the context when this special choice is made  and we will not mention it explicitly. (The symbol $\Omega$ will always denote a two-dimensional domain.) Let $Q=[-1/2,1/2]^n$ and for $\nu\in
S^{n-1}=\{x\in\R^n:|x|=1\}$, let $Q_\nu$ be a closed cube of sidelength one in $\R^n$, centered in the origin,  with one of its
sides parallel to $\nu$. (This defines $Q_\nu$ uniquely for $n=2$.) For a  set $K\in
\R^n$, $x_0\in\R^n$ and $\rho>0$,  we write $K(x_0,\rho)= x_0+\rho K$.
By $O(t)$, we denote terms $f(t)$ that satisfy $\limsup_{t\to \infty} t^{-1}|f(t)|<\infty$. We fix a radially symmetric function $\eta\in C^\infty_c(\R^n)$ such that $\supp \eta\subset B(0,1)$ and $\int_{\R^n}\eta\d x=1$, and define $\eta_\e:=\e^{-n}\eta(\cdot/\e)$.

\bigskip

The plan of the paper is as follows: In Section \ref{sec:preliminaries}, we will collect a number of theorems from the literature that we will apply later on. In Section \ref{sec:some-auxil-lemm}, we prove  some auxiliary lemmas, that will be used in Section \ref{sec:proof-main-theorem}, which contains the proof of the main theorem. In an attempt to increase readability, we have separated the part of the proof concerning the upper bound for points in the jump part from the rest  into a section on its own,  Section \ref{sec:proof-jump}. 
\section{Preliminaries}
\label{sec:preliminaries}
\subsection{Measures and BV functions}
Let $U\subset \R^n$ be open.

\begin{theorem}[Proposition 2.2 in \cite{ambrosio1992relaxation}]
\label{thm:RN}
  Let $\lambda, \mu$ be Radon measures in $U$ with $\mu\geq 0$. Then there
  exists a Borel set $E\subset U$ with $\mu(E)=0$ such that for any $x_0\in  \supp \mu\setminus E$ we have
\[
\lim_{\rho\downarrow 0} \frac{\lambda(x_0+\rho K)}{\mu(x_0+\rho
  K)}=\frac{\d\lambda}{\d\mu}(x_0)
\]
for any bounded convex set $K$ containing the origin. Here, the set $E$ is
independent of $K$.
\end{theorem}

\begin{theorem}[Theorem 2.3 in \cite{ambrosio1992relaxation}]
  \label{thm:BVblow}
Let $u\in BV(U;\R^m)$ and for a bounded convex open set $K$ containing the
origin, and let $\xi$ be the density of $Du$ with respect to $|Du|$,
$\xi=\frac{\d(Du)}{\d(|Du|)}$.
For $x_0\in \supp (|Du|)$, assume that $\xi(x_0)=\eta\otimes \nu$ with $\eta\in\R^m$,
$\nu\in\R^n$, $|\eta|=|\nu|=1$, and for
$\rho>0$ let 
\[
v^{(\rho)}(y)=\frac{\rho^{n-1}}{|Du|(x_0+\rho K)}\left(u(x_0+\rho
y)-\fint_{x_0+\rho K}u(x')\d x'\right)\,.
\]
Then for every $\sigma\in (0,1)$ there exists a sequence $\rho_j$ converging to
0 such that $v^{(\rho_j)}$ converges in $L^1(K;\R^m)$ to a function $v\in
BV(K;\R^m)$ which satisfies $|Dv|(\sigma \overline{K})\geq \sigma^n$ and can be
represented as 
\[
v(y)=\psi(y\cdot \nu)\eta
\]
for a suitable non-decreasing function $\psi:(a,b)\to \R$, where $a=\inf\{y\cdot
\nu:y\in K\}$ and $b=\sup\{y\cdot
\nu:y\in K\}$.
\end{theorem}

\medskip

When considering the blow-up of measures, the following special case of Theorem
0.1 in \cite{delladio1991lower} will be useful:

\begin{theorem}
  \label{thm:delladio}
  Let $\{\mu_j\}_j,\mu\in \mathcal M(U;\R^p)$, such that
  \[
    \mu_j\to \mu \quad \text{ and }\quad
    |\mu_j|\to|\mu|\quad\text{ weakly * in the sense of measures.}
  \]
  Furthermore, let $h:\R^p\to\R$ be positively one-homogeneous. Then
  \[
    h(\mu_j)\to h(\mu)  \quad\text{ weakly * in the sense of measures.}
    \]
\end{theorem}

We recall that $BH(U)$ denotes the set of functions $u\in W^{1,1}(U)$ such that
$\nabla u\in BV(U;\R^n)$.

The set 
$BH(U)$
can be made into a normed space by setting
\[
\|u\|_{BH(U)}=\|u\|_{W^{1,1}(U)}+|D\nabla u|(U)\,.
\]
We say that a sequence $u_j\in BH(U)$ converges weakly * to $u\in
BH(U)$ if $u_j\to u$ in $W^{1,1}(U)$ and $D\nabla u_j\to D\nabla  u$ weakly * in
$\M(U;\R^{n\times n})$.

\begin{theorem}[\cite{demengel1989compactness}]
\label{thm:BHcompact}
Let $u_j$ be a bounded
  sequence in $BH(U)$. Then there exists a subsequence (no relabeling) and
  $u\in BH(U)$ such that
\[
u_j\to u\quad\text{ weakly * in }BH(U)\,.
\]
\end{theorem}

Now let us assume that $U$ has smooth boundary. The trace operator
\[
\gamma_0:u\mapsto u|_{\partial U }
\]
is linear surjective as a map $W^{1,1}( U )\to L^1(\partial U )$ and also as a map $BV( U )\to
L^1(\partial  U )$. For the spaces
$W^{2,1}( U )$ and $BH( U )$, we may also  consider the operator
\[
\gamma_1:u\mapsto  \nabla u|_{\partial U }\cdot n\,,
\]
where $n$ denotes the unit outer normal of $\partial U $.
The following theorem combines statements from    Chapter 2 and  the appendix of \cite{demengel1984fonctions}.
\begin{theorem}
  \label{thm:traceop}
  \begin{itemize}
  \item[(i)] The operator 
$(\gamma_0,\gamma_1)$ is linear surjective  as a map 
\[
BH( U )\to \gamma_0(W^{2,1}( U ))\times L^1(\partial U )\,.
\]
\item[(ii)] There exists  a continuous right inverse 
\[
\gamma_0(W^{2,1}( U ))\times L^1(\partial U )\to W^{2,1}( U )\,.
\]
  \end{itemize}

\end{theorem}

For more on the space $BH$, see e.g.~\cite{savare1998superposition,fonseca2003lower}.

\subsection{Relaxation of integral functionals that depend on higher
  derivatives}
A
function $f:\R^{m\times n^k}\to \R$ is called $k$-quasiconvex if

\begin{equation}
f(\xi)=\inf\left\{\int_{[-1/2,1/2]^n}f(\xi+\nabla^k \varphi)\d x:\varphi\in
  W^{k,\infty}_0([-1/2,1/2]^n;\R^m)\right\}\,,\label{eq:23}
\end{equation}
see \cite{MR0188838}.  

\medskip

The so-called  $k$-quasiconvexification of $f:\R^{m\times n^k}\to \R$ is given by
the right hand side above,
\[
Q_kf(\xi)=\inf\left\{ \int_{[-1/2,1/2]^n}
f(\xi+\nabla^k\varphi)\d x:\,\varphi\in W^{k,\infty}_0([-1/2,1/2]^n;\R^m)\right\}\,.
\]
In the case $k=1$, one obtains the relaxation of integral
functionals $u\mapsto\int f(\nabla u)\d x$ by replacing $f$ by its
quasiconvex envelope $Q_1f$. 


\subsection{Blow-up method}

The main tool in our proof will be the so-called blow-up method. In the context
of lower semicontinuity of integral functionals in $BV$, this has been developed
by Fonseca and M\"uller.

\begin{theorem}[Theorem 2.19 in \cite{MR1218685}]
\label{thm:MFsingularpart} Let $f:\R^m\times \R^{m\times n}\to \R$ be 
   quasiconvex and positively one-homogeneous\footnote{When comparing our
     statement of the theorem with the one in \cite{MR1218685}, note that the assumption that $f$ is positively one-homogeneous implies that the
  recession function for $f$ is identical to $f$.} in the second argument. Assume that
 $v_j\to v$ weakly * in $BV(U)$ and  $f(v_j,\nabla
  v_j)\L^n\to \mu$ weakly * in the sense of measures, and that 
  $\zeta_2,\zeta_3$ are defined as the Radon-Nikodym derivatives
\[
\zeta_2=\frac{\d\mu}{\d(|D^sv|\ecke
C_v)},\quad \zeta_3=\frac{\d \mu}{\d (\H^1\ecke J_v)}\,.
\]
Then
  \[
  \begin{split}
    \zeta_2(x_0)&\geq f\left(v(x_0),\frac{\d Dv}{\d |Dv|}(x_0)\right)
    \quad\text{ for
    }|D^sv|\ecke C_{v}\text{ a.e. }x_0\in \Omega\\
    \zeta_3(x_0)&\geq K_f(v^+(x_0),v^-(x_0),\nu_v(x_0)) \quad\text{ for }|D^sv|\ecke
    J_{v}\text{ a.e. }x_0\in \Omega\,,
  \end{split}
  \]
  where
  \[
  \begin{split}
    K_f(a,b,\nu)=\inf\Big\{&\int_{Q_\nu}f(w,\nabla w)\d x: w\in W^{1,1}(Q_\nu), \\
    & w(x)=a \text{ for } x\cdot \nu=+1/2, \, w(x)=b \text{ for } x\cdot
    \nu=-1/2\Big\}\,.
  \end{split}
  \]
\end{theorem}

\section{Some auxiliary lemmas}
\label{sec:some-auxil-lemm}
\subsection{Relaxation and quasiconvexification}

 We consider the following integrands, defined for $v\in \R^2$, $\xi\in\Rsym$: 
\[
f_\lambda(v,\xi)=\lambda^{-1/2} F_\lambda(S(v,\xi))\sqrt{1+|v|^2}\,.
\]
This choice implies $\F_\lambda(u)=\int_\Omega f_\lambda(\nabla u,\nabla^2 u)\d x$.

\medskip

In order to find the lower semicontinuous envelope of $\F_\lambda$,  we will need to
determine  the 2-quasiconvexification of $f_\lambda$. In
principle this is contained in \cite{MR820342,allaire1993optimal}, and the
appendix of \cite{olbermann2017michell} contains a detailed proof of the case
$v=0$. Hence we only point out the modifications that are
necessary with respect to the latter; these changes can be found in the appendix to the
present paper.

\begin{proposition}
\label{prop:Q2F}
Let $v\in \R^{2}$.  The 2-quasiconvexification of $f_\lambda(v,\cdot)=\lambda^{-1/2}F_\lambda( S(v,\cdot))\sqrt{1+|v|^2}$ is given by

\begin{equation}
  Q_2f_\lambda(v,\xi)=\sqrt{1+|v|^2}\begin{cases}2 \rho^0(S(v,\xi))-2\frac{|\det S(v,\xi)|}{\sqrt{\lambda}} & \text{ if }\rho^0(S(v,\xi))\leq
  \sqrt{\lambda}\\
 \frac{|S(v,\xi)|^2}{\sqrt{\lambda}}+\sqrt{\lambda} & \text{ else.}\end{cases}\label{eq:1}
\end{equation}

\end{proposition}

In the sequel we use the notation 

  \begin{equation}
\begin{split}
  g_\lambda(\xi)&=\begin{cases} 2\left(\rho^0(\xi)-\frac{|\det \xi|}{\sqrt{\lambda}}\right)&\text{ if
    }\rho^0(\xi)\leq \sqrt{\lambda}\\
    \frac{|\xi|^2}{\sqrt{\lambda}}+\sqrt{\lambda}& \text{ else.}\end{cases}\\
  h_\lambda(v,\xi)&=Q_2 f_\lambda(v,\xi)
=g_\lambda(S(v,\xi))\sqrt{1+|v|^2}\,.
\end{split}\label{eq:55}
\end{equation}

\subsection{Properties of $h_\lambda$}

The following straightforward estimate will be used repeatedly:
\begin{lemma}
\label{lem:basicSestimate}
  Let $v\in \R^2,\xi\in \R^{2\times 2}$. Then
\[
|S(v,\xi)|^2\leq (1+|v|^2)^{-1}|\xi|^2\,.
\]
\end{lemma}
\begin{proof}
  This follows easily from the observation that $g(v)^{-1}$ is a symmetric matrix
  with eigenvalues $1$ and $(1+|v|^2)^{-1}$.
\end{proof}

\medskip

In the following lemma, we collect some properties of $g_\lambda$.

\begin{lemma}
\label{lem:hladd}
\begin{itemize}
\item[(i)] Let $M>1$. There exists a constant $C=C(M)$ such that whenever
  $A,B\in \Rsym$ with $|A|\leq M|B|$, we have
  \[
  g_\lambda(A)\leq C\, g_\lambda(B)\,.
  \]
\item[(ii)] For $A,B\in \Rsym$, we have
\[
|g_\lambda(A)-g_\lambda(B)|\leq C|A-B|\left(1+\frac{|A|+|B|}{\sqrt{\lambda}}\right)\,.
\]
\item[(iii)]  For every $\lambda>0$, we have
  \[
  g_\lambda(\xi)\geq 2 |\xi|_{\infty}\,.
  \]
\end{itemize}
\end{lemma}
\begin{proof}
We prove (i) by case distinction:
If $\sqrt{\lambda}\geq \rho^0(A)$, then we have
\[
  g_\lambda (A)\leq
   2\rho^0(A)\leq 4|A|\leq 4M|B|\leq 4M g_\lambda(B)\,.
\]
If $\rho^0(B)\leq \sqrt{\lambda}\leq \rho^0(A)$, then we have
\[
g_\lambda(A)= \sqrt{\lambda} +\frac{|A|^2}{\sqrt{\lambda}}\leq
2|A|+\frac{|A|^2}{M^{-1}|A|} \leq 3M^{2} |B|\leq 3M^2 g_\lambda(B)\,.
\]
If $\sqrt{\lambda}\leq \min(\rho^0(A),\rho^0(B))$, then 
\[
g_\lambda(A)=\sqrt{\lambda} +\frac{|A|^2}{\sqrt{\lambda}}
\leq M^2 g_\lambda(B).
\]
This completes the proof of (i).

\medskip

To prove (ii) it suffices to observe that $g_\lambda$ is piecewise
differentiable. A direct computation yields 
\[
|\nabla g_\lambda(A)|\leq C\left(1+\frac{|A|}{\sqrt{\lambda}}\right)
\]
almost everywhere, which immediately implies (ii).

\medskip

Finally we prove (iii). For $\xi=0$, the inequality is trivial. So let $\xi\neq
0$, and denote the eigenvalues of $\xi$ by $\tau_1,\tau_2$.  For $\rho^0(\xi)\leq\sqrt{\lambda}$, we have
\[
\begin{split}
  g_\lambda(\xi)&=2\left(\rho^0(\xi)-\frac{|\det \xi|}{\sqrt{\lambda}}\right)\\
  &\geq 2\left(\rho^0(\xi)-\frac{|\det \xi|}{|\xi|_\infty}\right)\\
  &\geq 2 \left(\rho^0(\xi)-\min(|\tau_1|,|\tau_2|)\right)\\
  &= 2 |\xi|_\infty\,.
\end{split}
\]
For $\rho^0(\xi)\geq\sqrt{\lambda}$, we have by the Cauchy-Schwarz inequality,
\[
  g_\lambda(\xi)=\sqrt{\lambda}+\frac{|\xi|^2}{\sqrt{\lambda}}
\geq 2\frac{|\xi|\sqrt{\lambda}}{\sqrt{\lambda}}
\geq 2|\xi|_\infty\,.
\]
This proves the lemma.
\end{proof}



In the following lemma, we introduce the following notation: The pointwise limit of $h_\lambda$ for $\lambda\to\infty$ is
\[
  G(v,\xi)=2 \rho^0(S(v,\xi))\sqrt{1+|v|^2}\,.
  \]

\begin{lemma}
\label{lem:hbounds}
We have that 
  \[
  \begin{split}
    \left|h_\lambda(v,\xi)-h_\lambda(\tilde v,\xi)\right|&\leq C|v-\tilde v|
    \max
    \left(h_\lambda(v,\xi),h_\lambda(\tilde v,\xi)\right) \\
    \left|f_\lambda(v,\xi)-f_\lambda(\tilde v,\xi)\right|&\leq C|v-\tilde v|
    \max \left(f_\lambda(v,\xi),f_\lambda(\tilde v,\xi)\right)\\
    \left|G(v,\xi)-G(\tilde v,\xi)\right|&\leq C|v-\tilde v|
    \max \left(G(v,\xi),G(\tilde v,\xi)\right) 
  \end{split}
\]
for all $v,\tilde v\in \R^2$, $\xi\in \Rsym$, where the constants $C$ do not
depend on $\lambda$.
\end{lemma}
\begin{proof}
We recall that $S(v,\xi)$ is given explicitly by 
\[
S(v,\xi)=(1+|v|^2)^{-1/2} \left(\id+v\otimes v\right)^{-1/2}\xi \left(\id+v\otimes
  v\right)^{-1/2}\,.
\]
We claim that
\begin{equation}
\label{eq:21}
\left|\nabla_v S(v,\xi)\right|\lesssim \frac{S(v,\xi)}{\sqrt{1+|v|^2}}\,.
\end{equation}
Indeed, 
noting that
\[
\left(\id+v\otimes v\right)^{-1/2}= \frac{1}{\sqrt{1+|v|^2}}\frac{v\otimes
  v}{|v|^2}+\frac{v^\bot\otimes v^\bot}{|v|^2}\,,
\]
this follows from a direct calculation, which we omit here.
Now we may estimate the partial derivative of $h_\lambda(v,\xi)$ using the chain
rule and Lemma \ref{lem:hladd} (ii),
\[
\begin{split}
  \left|\nabla_v h_\lambda(v,\xi)\right|&= \left|g_\lambda(S(v,\xi))\nabla_v
    \sqrt{1+|v|^2}
    + \sqrt{1+|v|^2} \nabla g_\lambda(S(v,\xi))\nabla_v S(v,\xi)\right|\\
  &\lesssim
  g_\lambda(S(v,\xi))+\sqrt{1+|v|^2}\left(1+\frac{S(v,\xi)}{\sqrt{\lambda}}\right)\frac{S(v,\xi)}{\sqrt{1+|v|^2}}\\
  &\lesssim g_\lambda(S(v,\xi))\\
  &\leq h_\lambda(S(v,\xi))\,.
\end{split}
\]
The analogous claim for $f_\lambda$ is trivial for $\xi=0$, and follows from
\eqref{eq:21} and the chain rule for $\xi\neq 0$. The inequality for $G$ is obtained from the one for $h_\lambda$ by taking the limit $\lambda\to \infty$.
\end{proof}

The following lemma will provide the proof of the lower bound once the
additional complication of the lower
order terms has been treated. 

\begin{lemma}
\label{lem:qcL1conv}  
Let $\Omega\subset \R^2$ be open and bounded, $v_0\in\R^2$, $\xi_0\in \R^{2\times 2}$,
  $w_\lambda\to 0$ in $L^1(\Omega)$ as $\lambda\to \infty$, and $\|\nabla w_\lambda\|_{L^1}<C$. Then
\[
\liminf_{\lambda\to\infty}\int_\Omega h_\lambda(v_0,\xi_0+\nabla w_\lambda)\d x\geq
2\L^2(\Omega)\rho^0(S(v_0,\xi_0))\sqrt{1+|v_0|^2}\,.
\]
\end{lemma}

\begin{proof}
Up to details, the proof is identical to the proof of Lemma 6.2 (i)
in  \cite{olbermann2017michell}. There it is proved that
\[
\liminf_{\lambda\to\infty}\int_\Omega g_\lambda(\xi_0+\nabla w_\lambda)\d x\geq
2\L^2(\Omega)\rho^0(\xi_0)\,.
\]
In that proof, one only needs to replace $g_\lambda$ with
$g_\lambda(S(v_0,\cdot))$. Apart from the additional dependence of some of the
constants ``$C$'' on $v_0$ that appear in the proof, all arguments go through unchanged.
\end{proof}

\subsection{Blow-up of higher order gradients}

Theorem \ref{thm:MFsingularpart} describes the behavior of integrands depending on gradients under
the blow-up procedure. This  will not be quite enough for our purposes: For the jump part, our
proof will take advantage of the fact that we consider the second fundamental form of
the graph, which in turn means that we need to consider integrands that depend on first and
second derivatives.

\begin{lemma}
  \label{lem:MFvariant}
  Let $\Omega\subset\R^2$ be open and bounded.
Assume that $f:\R^2\times \R^{2\times 2}\to \R$ fulfills the following
properties:
\begin{itemize}
\item[(i)] $f$ is quasiconvex and positively one-homogeneous in the second
  argument with $f(v,\xi)\leq C|\xi|$
\item[(ii)] The functional $u\mapsto
  \int_{\Omega} f(\nabla u,\nabla^2 u)\d x$ is continuous in $W^{2,2}(\Omega)$
\end{itemize}
Furthermore assume that $u_\lambda$ is a sequence in $W^{2,2}(\Omega)$, 
 $u_\lambda\to u$ weakly * in $BH(\Omega)$,   $f(\nabla u_j,\nabla^2
  u_j)\L^n\to \mu$ weakly * in the sense of measures, and that 
  $\zeta_3$ is defined as the Radon-Nikodym derivative
\[
 \zeta_3=\frac{\d \mu}{\d (\H^1\ecke J_{\nabla u})}\,.
\]
Then
  \[
  \begin{split}
    \zeta_3(x_0)&\geq \tilde K_f(\nabla u^+(x_0),\nabla u^-(x_0),\nu_{\nabla u}(x_0))
    \quad\text{ for }|D^s\nabla u|\ecke
    J_{\nabla u}\text{ a.e. }x_0\in  \Omega \,,
  \end{split}
  \]
  where
  \[
  \begin{split}
    \tilde K_f(a,b,\nu)=\inf\left\{\int_{Q_\nu}f(\nabla w,\nabla^2 w)\d x: w\in
    \mathcal A_{a,b,\nu}\right\}
\end{split}
  \]
and
\[
\begin{split}
  \mathcal A_{a,b,\nu}&=\Bigg\{ w\in
  C^\infty(\overline{Q_\nu}):\\
&\qquad  w(x)=a\cdot x \text{ in some neighborhood of
  }\left\{x\in \partial Q_\nu:x\cdot\nu=\frac12\right\}\,,\\
&\qquad  w(x)=b\cdot x \text{ in some neighborhood of
  }\left\{x\in \partial Q_\nu:x\cdot\nu=-\frac12\right\}\,,\\
&\qquad  \nabla^k w(x+\nu^\bot)=\nabla^k w(x) \text{ for  } x\cdot
\nu^\bot=-\frac12 \text{ and } k=1,2,\dots\Bigg\}\,.
\end{split}
\]

\end{lemma}

\begin{proof}

We write $\nu\equiv \nu_{\nabla u}(x_0)$.
With 
\[
u_\lambda^{(\rho)}(x)=\rho^{-1} \left(u_\lambda(x_0+\rho
  x)-u_\lambda(x_0)\right)\,,\qquad  U(x)= \begin{cases}\nabla u^+(x_0)\cdot x &\text{
    if }x\cdot\nu\geq 0\\\nabla u^-(x_0)\cdot x&\text{ if } x\cdot \nu<
  0\,,\end{cases}
\]
we have that for $|D\nabla u|\ecke J_{\nabla u}$ almost every $x_0$,
$\lim_{\rho\to 0}\lim_{\lambda\to\infty}u_\lambda^{(\rho)}= U$ in $W^{1,1}(Q_\nu)$, see Theorem 3.77
in
\cite{MR1857292}. Additionally,  
\[
\zeta_3(x_0)=\lim_{\rho\to 0}\lim_{\lambda\to
  \infty}\rho^{-1}\int_{Q_\nu(x_0,\rho)}f(\nabla u_\lambda,\nabla^2
u_\lambda)\d x\,.
\]
Choose $\rho_j\to 0,\lambda_j\to \infty$ such that $u_{\lambda_j}^{(\rho_j)}\to
U$ in $W^{1,1}(Q_\nu)$ and 
\[
\begin{split}
  \zeta_3(x_0)&=\lim_{j\to\infty}\rho_j^{-1}\int_{Q_\nu(x_0,\rho_j)}f(\nabla
  u_{\lambda_j},\nabla^2
  u_{\lambda_j})\d x\,.
\end{split}
\]
We write
$u_j:=u_{\lambda_j}^{(\rho_j)}$. 
We set $U_j:=\eta_{\rho_j}*U$. $U_j$ is affine on the slices orthogonal to $\nu$.
With this notation, we have
\[
  \zeta_3(x_0)=\lim_{j\to\infty}\int_{Q_\nu}f(\nabla
  u_{j},\nabla^2
  u_j)\d x\,. 
  \]
Hence it remains  to show
\begin{equation}
\tilde K_f(\nabla u^+(x_0),\nabla u^-(x_0),\nu)\leq \lim_{j\to\infty}\int_{Q_\nu}f(\nabla
  u_{j},\nabla^2
  u_j)\d x\,.\label{eq:32}
  \end{equation}
By the continuity assumption (ii), we may assume that $u_j\in C^\infty(\overline{Q_\nu})$
in the proof of \eqref{eq:32}.

\medskip

For $l\in \N$, let $K_l\in\N$ be the smallest integer that satisfies
\[
K_l>l\sup_j\left\{\|u_j\|_{W^{2,1}}+\|U_j\|_{W^{2,1}}\right\}\,,
\]
and
\[
\alpha_l:=\max\left(\frac{1}{l},\sup\{\|u_j-U_j\|_{W^{1,1}}:j>l\}\right)\,,\qquad
s_l:=\frac{\alpha_l}{K_l}\,.
\]
Note that $\alpha_l\to 0$ as $l\to \infty$.
For $i=0,\dots,K_l$, let
\[
Q_{i,l}=(1-\alpha_l+i\, s_l)Q_\nu\,.
\]
Consider a family of cut-off functions $\{\varphi_{i,l}:i=1,\dots,K_l\}$ with 
\[
\varphi_{i,l}\in C_c^\infty(Q_{i,l})\,,\quad 0\leq \varphi_{i,l}\leq 1\,,\quad\varphi_{i,l}=1
\text{ on }
Q_{i-1,l}\,,\quad\|\nabla^k \varphi_{i,l}\|_{L^\infty}=O(s_l^{-k})\text{ for } k=1,2\,.
\]
For $j>l$, we define 
\[
\tilde u_j^{i,l}:=\varphi_{i,l} u_j+(1-\varphi_{i,l}) U_j\,.
\]
We have that $\tilde u_j^{i,l}\in \mathcal A_{\nabla u^+(x_0),\nabla
  u^-(x_0),\nu}$ (for $j$ large enough).
On $Q_{i,l}\setminus Q_{i-1,l}$, we have 
\[
\begin{split}
  \nabla^2 \tilde u_j^{i,l}&= (u_j-U_j)\nabla^2\varphi_{i,l}+\nabla
  (u_j-U_j)\otimes \nabla\varphi_{i,l}\\
  &\quad + \nabla\varphi_{i,l}\otimes \nabla
  (u_j-U_j)+\varphi_{i,l}\nabla^2(u_j-U_j)+\nabla^2 U_j\,.
\end{split}
\]

Now we may estimate, for every $i=1,\dots, K_l$,
\begin{equation}
\begin{split}
  \int_{Q_\nu} f(\nabla \tilde u_j^{i,l},\nabla^2 \tilde u_j^{i,l})\d x&
  \leq \int_{Q_{i-1,l}} f(\nabla u_j,\nabla^2 u_j)\d x
 + C\int_{Q_{i,l}\setminus Q_{i-1,l}}|\nabla^2 \tilde u_j^{i,l}|\d x\\
&\quad + C\int_{Q_\nu\setminus Q_{i,l}}|\nabla^2 U_j|\d x\\
&\leq  \int_{Q_{i-1,l}} f(\nabla u_j,\nabla^2 u_j)\d x\\
&\quad + C\int_{Q_{i,l}\setminus Q_{i-1,l}}s_l^{-2}|u_j-U_j|+s_l^{-1}|\nabla u_j-\nabla
 U_j|+|\nabla^2 u_j|+|\nabla^2 U_j|\d x\\
&\quad + C\int_{Q_\nu\setminus Q_{i,l}}|\nabla^2 U_j|\d x\\
\end{split}\label{eq:5}
\end{equation}
We write $T_{i,l}=Q_{i,l}\setminus Q_{i-1,l}$, and  choose an increasing sequence $j(l)$ with $j(l)>l$ such that for every $i=1,\dots,K_l$,
\[
\begin{split}
  \fint_{T_{i,l}} \left|u_j-U_j\right|\d x&<s_l^2\\
  \fint_{T_{i,l}} \left|\nabla u_j-\nabla U_j\right|\d x&
  <s_l\,.
\end{split}
\]
This is possible by $\|u_j-U_j\|_{W^{1,1}}\to 0$.
With the help of these estimates,
the second error term in \eqref{eq:5} for $j=j(l)$ can be estimated as follows,
\[
\begin{split}
  \int_{T_{i,l}}&s_l^{-2}|u_j-U_j|+s_l^{-1}|\nabla u_j-\nabla
  U_j|+|\nabla^2 u_j|+|\nabla^2 U_j|\d x\\
    &\leq C \left(\|\nabla^2 u_j\|_{L^1(T_{i,l})}+ \|\nabla^2 U_j\|_{L^1(T_{i,l})}+s_l\right)\,.
  \end{split}
\]
Summing over all $i$ and averaging, we obtain
\[
\begin{split}
  \frac{1}{K_l}\sum_{i=1}^{K_l} \int_{Q_\nu} f(\nabla \tilde u_j^{i,l},\nabla^2
  \tilde u_j^{i,l})\d x &\leq \int_{Q_\nu} f(\nabla u_j,\nabla^2 u_j)\d x
  + \frac{C}{K_l}\int_{Q_\nu}|\nabla^2 U_j|\d x\\
  &\quad +\frac{C}{K_l}\int_{Q_\nu}(|\nabla^2 u_j|+|\nabla^2 U_j|+1)\d x+C s_l
\end{split}
\]
Since the error terms vanish for $l\to \infty$, we can choose
$i=i(l)\in\{1,\dots,K_l\}$ such that
\[
\liminf_{l\to \infty}\int_{Q_\nu} f(\nabla \tilde u_j^{i,l},\nabla^2
  \tilde u_j^{i,l})\d x
\leq \lim_{j\to \infty}\int_{Q_\nu} f(\nabla  u_j,\nabla^2
   u_j)\d x\,.
\]
Since $\tilde u_j^{i,l}\in \mathcal A_{\nabla u^+(x_0),\nabla u^-(x_0),\nu}$, the last equation proves
\eqref{eq:32}.
\end{proof}

\subsection{Geometric considerations}
\label{sec:geometry}

We will need to apply Lemma \ref{lem:MFvariant} to the following particular
choice of integrand:
\[
G_\infty(v,\xi)=2\left| S(v,\xi)\right|_{\infty} \sqrt{1+|v|^2}\,.
\]

By some geometric considerations, we are able to determine  $\tilde
K_{G_\infty}$ in Lemma \ref{lem:Kcalc} below. We start with a preparatory
lemma. The assumptions are chosen such that we may apply the lemma to graphs of functions
in $\A_{a,b,\nu}$ as defined in Lemma \ref{lem:MFvariant} with $\nu=e_2$, see Figure \ref{fig:norot}.

\begin{lemma}
  \label{lem:slice}
Let $M$ be an oriented $C^2$ submanifold of $\R^3$ with the following
properties:
\begin{itemize}
\item[(i)] $M$ is diffeomorphic to a square
\item[(ii)] There exists $l>0$ and for each $x_1\in[0,l]$ there exists a $C^2$
  curve $\gamma_{x_1}$ contained in $\{x_1\}\times[0,1]\times\R$ with its two endpoints
  in $\{x_1\}\times\{0\}\times\R$ and $\{x_1\}\times\{1\}\times\R$ respectively, such that
\[
M=\bigcup_{x_1\in[0,l]}\gamma_{x_1}\,.
\]
\item[(iii)] There exist $\NN_0,\NN_1\in S^2$ such that the for each $x_1\in
  [0,l]$, the surface normals in the endpoints of $\gamma_{x_1}$ are given by
  $\NN_0,\NN_1$ respectively.
\end{itemize}
Then 
\[
\int_{M}|S_{M}|_\infty\d\H^2\geq l \arccos \NN_0\cdot\NN_1\,,
\]
and equality holds if  any two  curves $\gamma_{x_1}, \gamma_{x_1'}$ are
parallel translations of each other
in $x_1$ direction, and
their curvature does not change sign.
\end{lemma}
\begin{proof}
Looking at the slices for  $x_1=$constant,  we have that
\[
  \int_{M}|S_M|_\infty\d\H^2 \geq \int_0^l
  \int_{\gamma_{x_1}}|S_M|_\infty\d\H^1\,.
\]
Denoting by $N_{x_1}$ a differentiable choice of a normal to $M$ along $\gamma_{x_1}$, we have that the derivative of the normal
$DN_{x_1}$ fulfills
\[|DN_{x_1}|_\infty\leq
|S_M|_\infty\,.
\] 
Hence, by the fundamental theorem of calculus, and letting
$\dist_{S^2}(\cdot,\cdot)$ denote the geodesic distance on $S^2$,
\[
\int_{\gamma_{x_1}}|S_M|_\infty\d\H^1\geq 
  \int_{\gamma_{x_1}}|D N_{x_1}|\d\H^1\geq \dist_{S^2}(\NN_0,\NN_1)=
 \arccos \NN_0\cdot \NN_1\,.
\]
The claimed inequality follows. If the curves $\gamma_{x_1}$ are parallel translations of
each other in $x_1$-direction  and
their curvature does not change sign, then the inequalities become sharp.
\end{proof}

\begin{figure}[h]
\begin{subfigure}{.45\textwidth}
\includegraphics[height=5cm]{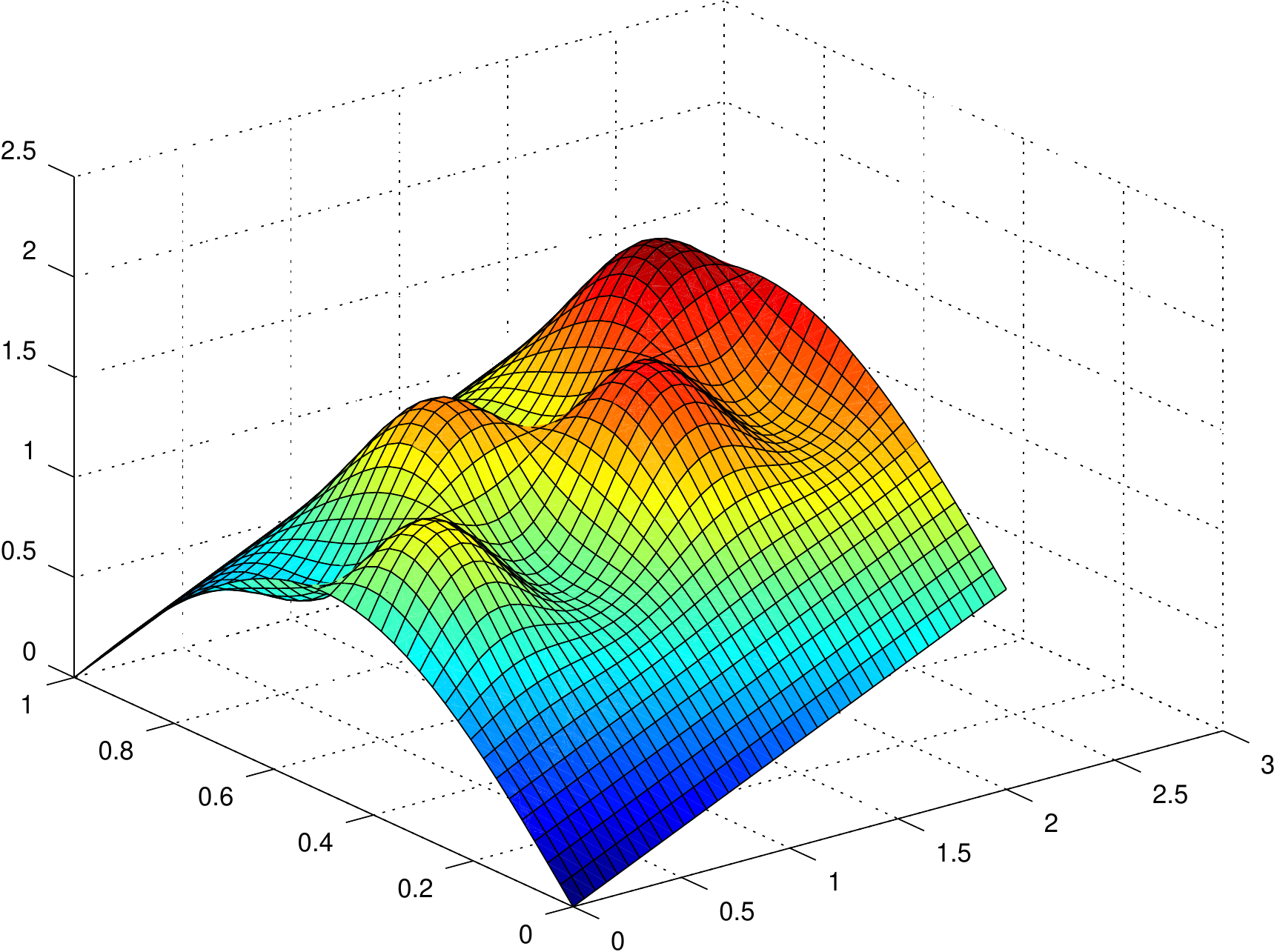}
\caption{The  graph of some function $w$ to which Lemma \ref{lem:slice} may be applied.\label{fig:norot}}
\end{subfigure}
\hspace{5mm}
\begin{subfigure}{.45\textwidth}
\includegraphics[height=5cm]{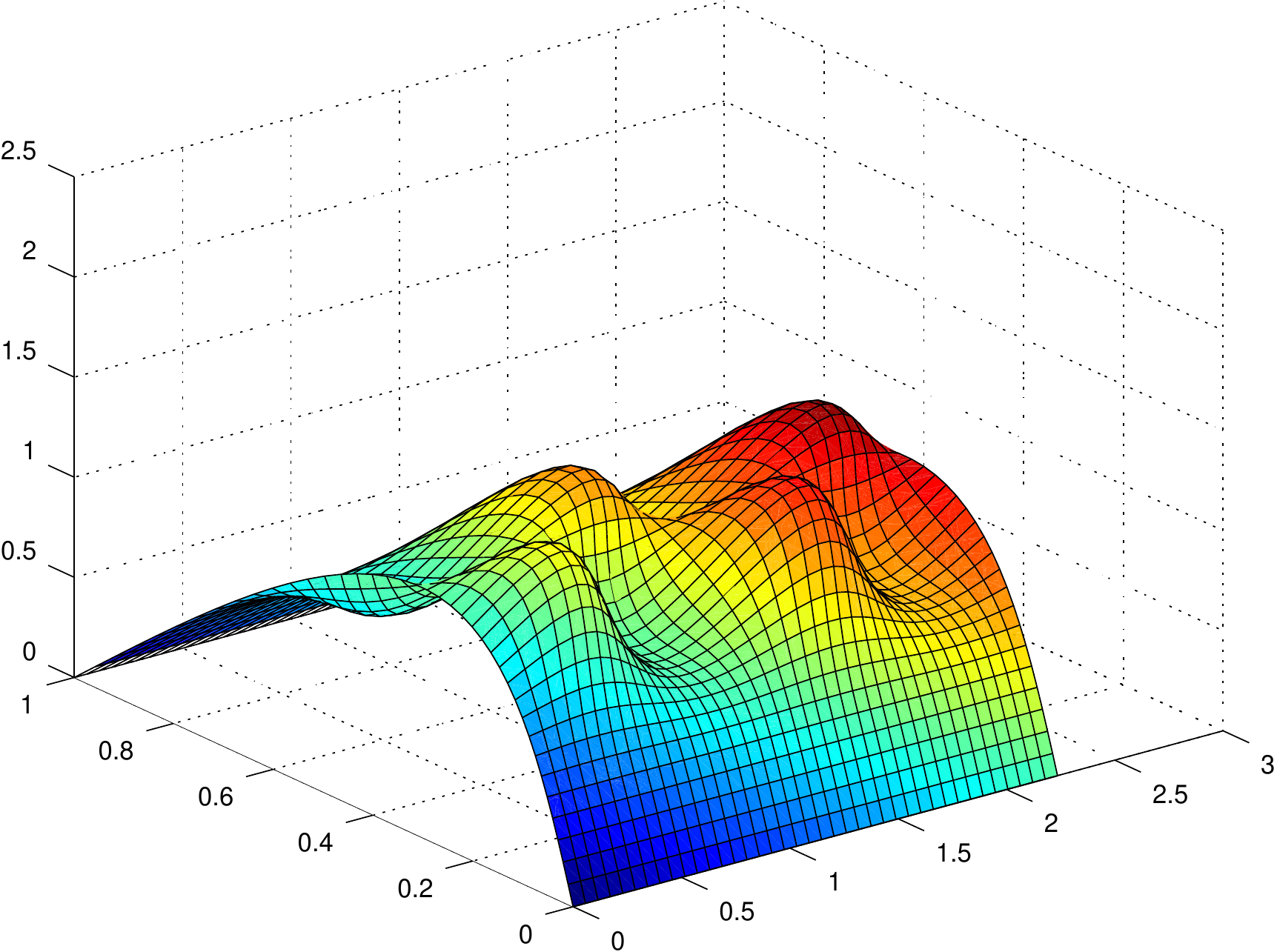}
\caption{After applying a suitable Euclidean motion $R$, we have $R(\gr \, w|_{[0,l]\times\{0\}})\subset \R\times
  \{(0,0)\}$ and $R(\gr \, w|_{[0,1]\times\{1\}})\subset \R\times
  \{(1,0)\}$. \label{fig:rot}} 
\end{subfigure}
\end{figure}

\begin{lemma}
\label{lem:Kcalc}
  Let $a,b\in \R^2$, $\nu\in S^1$ with $a\cdot \nu^\bot=b\cdot\nu^\bot$, and
  $G_\infty(v,\xi)=2|S(v,\xi)|_\infty\sqrt{1+|v|^2}$. Then with $\tilde K$
  defined as in the statement of Lemma \ref{lem:MFvariant}, we have that
\[
\tilde K_{G_\infty}(a,b,\nu)=2\sqrt{1+|a\cdot \nu^\bot|^2}\arccos \NN(a)\cdot \NN(b) \,.
\]
\end{lemma}
\begin{proof}
Let $w\in \A_{a,b,\nu}$. After a rotation of the coordinate system, we may
assume that $\nu= e_2$ and $a_1=b_1$. Let $M_1$ denote the graph of $w$. By
applying a suitable 
Euclidean motion (namely, a rotation
 with axis parallel to $e_2$ and a translation), we may map $\gr\, w|_{[0,1]\times \{0\}}$ to
$[0,\sqrt{1+a_1^2}]\times\{(0,0)\}$ and   $\gr\, w|_{[0,1]\times \{1\}}$ to
$[0,\sqrt{1+a_1^2}]\times\{(1,0)\}$ respectively, see Figure \ref{fig:rot}. Let us denote the
resulting submanifold of $\R^3$ by $M_2$. By the periodicity 
of $\nabla^k w$ for $k\in \{1,2\}$ in $x_1$-direction, we may translate $M_2\cap [0,l]\times[0,1]\times\R$ in $x_1$-direction by
$l=\sqrt{1+a_1^2}$, and the resulting set will still be a $C^2$ submanifold,  with $\int_{M_3}|S_{M_3}|_\infty\d\H^2=\int_{M_1}|S_{M_1}|_\infty\d\H^2$. To
$M_3$, we may apply Lemma \ref{lem:slice} to obtain the claimed lower bound. If $\nabla w$
is constant in $x_1$ direction and $e_2\cdot \nabla w$ is monotone in $x_2$ direction, the second
part of that lemma yields that the bound is also attained.
\end{proof}

\section{Proof of the main theorem}
\label{sec:proof-main-theorem}
\subsection{Compactness}

\begin{proof}[Proof of Theorem \ref{thm:main} (i)]
Using $\|\nabla u_\lambda\|_{L^\infty}<C$, we have that 
\begin{equation}
|\nabla^2 u_\lambda|\leq C |S(\nabla
u_\lambda,\nabla^2 u_\lambda)|\label{eq:10}
\end{equation}
By Lemma \ref{lem:hladd} (iii), we have that
\begin{equation}
|\xi|\leq g_\lambda(\xi)\quad \text{ for all } \xi \in \Rsym\,.\label{eq:33}
\end{equation}
From \eqref{eq:10} and \eqref{eq:33} it follows that
\[
|\nabla^2 u_\lambda|\leq h_\lambda(\nabla u_\lambda,\nabla^2 u_\lambda)\,,
\]
and hence
\[
\limsup \|\nabla^2 u_\lambda\|_{L^1(\Omega)}\leq C\,.
\]
By Theorem \ref{thm:BHcompact}, we obtain the weak * convergence in $BH$ for
a subsequence.
\end{proof}

\subsection{Lower bound}

\begin{proof}[Proof of Theorem \ref{thm:main} (ii)]
  The main tool of the proof is the blowup technique by Fonseca and M\"uller. We
  have that
  \[
  D\nabla u= \nabla^2 u \L^2 + D^s\nabla u\ecke C_{\nabla u}+(\nabla u^+-\nabla
  u^-)\otimes \nu_{\nabla u}\H^1\ecke
  J_{\nabla u}\,.
  \]
In the sequel, we write $\nu\equiv \nu_{\nabla u}$.
  After choosing a subsequence, we may assume that $\lim_{\lambda\to
    \infty}\F_\lambda(u_\lambda)=\liminf_{\lambda\to\infty}\F_\lambda(u_\lambda)$,
  without increasing the $\liminf$.  Recalling the definition \eqref{eq:55} of
  $h_\lambda$, we have that $h_\lambda=Q_2f_\lambda\leq
  f_\lambda$, and hence there exists a Radon measure $\mu$ such that (after passing to a
  further subsequence)
  \[
  h_\lambda(\nabla u_\lambda,\nabla^2 u_\lambda)\L^2\to \mu\quad\text{ weakly *
    in the sense of measures.}
  \]
  Let $\zeta_1,\zeta_2,\zeta_3$ denote the Radon-Nikodym derivative of $\mu$
  with respect to $\L^2$, $|D^s\nabla u|\ecke C_{\nabla u}$ and $\H^1\ecke
  J_{\nabla u}$ respectively. By the non-negativity of $\mu$, we have
  \[
  \mu\geq \zeta_1 \L^2 +\zeta_2 |D^s\nabla u|\ecke C_{\nabla u}+ \zeta_3
  \H^1\ecke J_{\nabla u} \,.
  \]
  We will show that

\begin{align}
  \zeta_1(x)\geq &2\rho^0\left(S\left(\nabla u(x),\nabla^2
      u(x)\right)\right)\sqrt{1+|\nabla u|^2} \quad
  \text{ for }\L^2-\text{a.e.}\, x\in \Omega \label{eq:2}\\
  \zeta_2(x)\geq &2\rho^0\left(S\left(\nabla u(x),\frac{\d (D\nabla u)}{\d
        |D\nabla u|}(x)\right)\right)\sqrt{1+|\nabla u|^2} \quad \text{ for
  }|D^s\nabla u|-\text{a.e.}\, x\in C_{\nabla u}\label{eq:4}\\
  \zeta_3(x)\geq &2\arccos\left(\NN(\nabla u^+)\cdot \NN(\nabla
    u^-)\right)\sqrt{1+|\nu^\bot\cdot\nabla u|^2} \label{eq:3}\quad \text{ for
  }\H^1-\text{a.e.}\, x\in J_{\nabla u}\,.
\end{align}
This will prove the lower bound.

We will first prove \eqref{eq:2}.  

We write $v_\lambda=\nabla u_\lambda$.  
 For $\L^2$-almost every $x_0$, we may
choose a sequence $(\e_j)_{j\in\N}$ converging to zero, such that $\mu(\partial
Q(x_0,\e_j))=0$ for every $j\in \N$. When we write $\e\to 0$ in the sequel, we
actually mean the limit $j\to\infty$ for such a sequence. Also, we will drop the index $j$ in our notation. For every $\e$, we have
\[
\lim_{\lambda\to\infty}\int_{Q(x_0,\e)} h_\lambda(v_\lambda,\nabla
v_\lambda)\d x = \mu(Q(x_0,\e))\,.
\]
Moreover,
\[
\lim_{\e\to 0}\lim_{\lambda\to\infty} \frac{ Dv_\lambda(Q(x_0,\e))}{|D
  v|(Q(x_0,\e))}=\frac{\d D v}{\d |D v|}(x_0)\,.
\]
Note that by Theorem \ref{thm:RN} we have
\begin{equation}
  \begin{split}
    \zeta_1(x_0)=&\lim_{\e\to 0} \frac{\mu(Q({x_0,\e}))}{\L^2(Q(x_0,\e))}\\
    =&\lim_{\e\to 0}\lim_{\lambda\to
      \infty}\fint_{Q(x_0,\e)}h_\lambda(v_\lambda,\nabla v_\lambda)\d x\,.
  \end{split}\label{eq:12}
\end{equation}
We write $v_0:=v(x_0)$. For $\e$ small enough, define $w_{\lambda,\e}:Q\to \R^2$ by
\[
w_{\lambda,\e}(x)=\e^{-1}\left(v_\lambda(x_0+\e x)-v_0\right)\,.
\]
Furthermore let $w_0(x)=\nabla v_0\cdot x$. Using a change of variables and the
Cauchy-Schwarz inequality we have
\begin{equation}
  \begin{split}
    \lim_{\e\to 0}\lim_{\lambda\to\infty} \|w_{\lambda,\e}-w_0\|_{L^1(Q)}&=
    \lim_{\e\to 0} \frac{1}{\e}\int_{Q}|v(x_0+\e x)-v_0-\nabla v_0\cdot
    \e x|\d x\\
    &= \lim_{\e\to 0} \frac{1}{\e^3}\int_{Q(x_0,\e)}|v(x)-v_0-\nabla
    v_0\cdot
    (x-x_0)|\d x\\
    &\leq \lim_{\e\to 0}\frac{1}{\e^2}\left(\int_{Q(x_0,\e)}|v(x)-v_0-\nabla
      v_0\cdot
      (x-x_0)|^2\d x\right)^{1/2}\\
    &=0\,.
  \end{split}\label{eq:11}
\end{equation}
The last equality above holds for $\L^2$-almost every
$x_0$ by the remark below Theorem 3.83 in
\cite{MR1857292} (which is a slightly stronger variant of approximate differentiability). 
Also note that we have
\[
\fint_{Q(x_0,\e)}h_\lambda(v_\lambda,\nabla v_\lambda)\d
x=\int_{Q}h_\lambda(v_0+\e w_{\lambda,\e}(x),\nabla w_{\lambda,\e}(x))\d x\,.
\]

By \eqref{eq:12} and \eqref{eq:11}, it is possible to choose $\e\equiv \e(\lambda)$ depending on $\lambda$ such that
with $w_\lambda:=w_{\lambda,\e(\lambda)}$ we have that
\[
\begin{split}
  \lim_{\lambda\to \infty}\|w_\lambda-w_0\|&=0\\
  \lim_{\lambda\to \infty}\int_{Q}h_{\lambda}(v_0+\e w_\lambda(x),\nabla w_\lambda)\d
  x&=\zeta_1(x_0)\,.
\end{split}
\]

We need to modify $w_\lambda$ such that we get a suitable $L^\infty$-bound for fixing
the lower order terms. Namely, we are going to construct $\tilde w_\lambda$ such that
$\|\tilde w_\lambda\|_{L^\infty}\leq \e^{-1/2}$, and

\begin{equation}
  \liminf_{\lambda\to\infty}\int_Q h_{\lambda}(v_0+\e\tilde w_\lambda,\nabla \tilde w_\lambda)\d
  x\leq 
  \liminf_{\lambda\to\infty}\int_Q h_{\lambda}(v_0+\e w_\lambda,\nabla  w_\lambda)\d
  x\,.\label{eq:14}
\end{equation}

Let $K_\lambda$ be the largest integer smaller than $\log_2 \e^{-1/2}$. For
$k=1,\dots,K_\lambda$, we define
\[
E_k^\lambda:=\left\{x\in Q:\,2^{k-1}<|w_\lambda-w_0|\leq 2^k\right\}\,.
\]
Next we choose $k_\lambda\in\{1,\dots, K_\lambda\}$ such that with $E^\lambda:=E_{k_\lambda}^\lambda$, we have
\begin{equation}
  \int_{E^\lambda}\left(1+h_{\lambda}(v_0+\e w_\lambda,\nabla  w_\lambda)\right)\d
  x\leq K_\lambda^{-1}\int_Q \left(1+h_{\lambda}(v_0+\e w_\lambda,\nabla  w_\lambda)\right)\d
  x\label{eq:15}
\end{equation}
and we define $\varphi_\lambda:[0,\infty)\to\R$ such that
\[
\begin{split}
  \varphi_\lambda(x)=&1\text{ for } x\leq 2^{k_\lambda-1}\\
  \varphi_\lambda(x)=&0\text{ for } x\geq 2^{k_\lambda} \\
  |\varphi_\lambda'|\leq &2^{2-k_\lambda}\,.
\end{split}
\]
Now we set
\[
\tilde w_\lambda=w_0+\varphi_\lambda(|w_\lambda-w_0|)(w_\lambda-w_0)\,.
\]
Note that $\|\tilde w_\lambda-w_0\|_{L^\infty}\leq \e^{-1/2}$ by
construction, 
and $\tilde w_\lambda=w_0$ on
$\{x\in\Omega:|w_\lambda-w_0|\geq 2^{k_\lambda}\}$. We have that
\begin{equation}
  \begin{split}
    \liminf_{\lambda\to\infty}\int_Q h_{\lambda}(v_0+\e\tilde w_\lambda,\nabla \tilde
    w_\lambda)\d x&\leq \liminf_{\lambda\to\infty}\int_{\{|w_\lambda-w_0|\leq 2^{k_\lambda-1}\}}
    h_{\lambda}(v_0+\e w_\lambda,\nabla w_\lambda)\d
    x\\
    &\quad +  \int_{E^\lambda} h_{\lambda}(v_0+\e \tilde w_\lambda,\nabla \tilde w_\lambda)\d x\\
    &\quad + \int_{\{|w_\lambda-w_0|\geq 2^{k_\lambda}\}} h_{\lambda}(v_0+\e w_0,\nabla
    w_0)\d x\,.
  \end{split}
  \label{eq:35}
\end{equation}

We claim that

\begin{equation}\label{eq:24}
  \int_{E^\lambda} h_{\lambda}(v_0+\e \tilde w_\lambda,\nabla \tilde w_\lambda)\d x \leq
  C(v_0,\nabla w_0) \int_{E^\lambda} \left(1+h_{\lambda}(v_0+\e w_\lambda,\nabla w_\lambda)\right)\d x\,.
\end{equation}
Indeed, we have that on $E^\lambda$, $|\e w_\lambda-\e\tilde w_\lambda|\lesssim
\e^{1/2}$, and hence
\begin{equation}
  \begin{split}
    |\gm(v_0+\e w_\lambda)-\gm(v_0+\e\tilde w_\lambda)|&\leq C(v_0) \e^{1/2}\,,\\
    \left|\sqrt{1+|v_0+\e w_\lambda|^2}-\sqrt{1+|v_0+\e\tilde w_\lambda|^2}\right|&\leq
    C \e^{1/2}\,.
  \end{split}\label{eq:34}
\end{equation}
Also,
\[
\begin{split}
  |\nabla \tilde w_\lambda|&= \left|\nabla w_0+\varphi_\lambda' (w_\lambda-w_0)\otimes \nabla
    |w_\lambda-w_0|+\varphi_\lambda \nabla (w_\lambda-w_0)\right|\\
  &\leq C\left( |\nabla w_0|+ (2^{-k_\lambda}|w_\lambda-w_0|+1)|\nabla (w_\lambda-
    w_0)|\right)\\
  &\leq C(\nabla w_0) (|\nabla w_\lambda|+1)\,.
\end{split}
\]
Hence, we have that
\[
|S(v_0+\e\tilde w_\lambda, \nabla \tilde w_\lambda)|\leq C(v_0,\nabla w_0)\left(|S(v_0+\e w_\lambda,
  \nabla  w_\lambda)|+1\right)
\]
and it follows from Lemma \ref{lem:hladd} (i) that

\begin{equation}
  g_{\lambda}\left(S(v_0+\e\tilde w_\lambda, \nabla \tilde w_\lambda)\right)
  \leq 
  C (v_0,\nabla w_0)\left(g_{\lambda}(S(v_0+\e\tilde w_\lambda,
    \nabla  w_\lambda))+1\right)\,.\label{eq:25}
\end{equation}

Our claimed inequality \eqref{eq:24} now follows from \eqref{eq:34} and
\eqref{eq:25}.

Using \eqref{eq:15}, \eqref{eq:24} and the fact $K_\lambda\to\infty$ as $\lambda\to \infty$,
as well as $w_\lambda\to w_0$ in $L^1$, the right hand side of \eqref{eq:35} can be
estimated from above by
\[
\begin{split}
  \liminf_{\lambda\to\infty}\left(\int_{Q} h_{\lambda}(v_0+\e w_\lambda,\nabla w_\lambda)\d x
    + C(v_0,\nabla w_0)\L^2\left(\{|w_\lambda-w_0|\geq 2^{k_\lambda}\}\right)
    |\nabla w_0|\right)\\
  = \liminf_{\lambda\to\infty}\int_{Q} h_{\lambda}(v_0+\e w_\lambda,\nabla w_\lambda)\d x\,.
\end{split}
\]
This proves \eqref{eq:14}.

Now we have by Lemma \ref{lem:hbounds} and Lemma \ref{lem:qcL1conv},
\[
\begin{split}
  \liminf_{\lambda\to\infty}\int_Q h_{\lambda}(v_0+\e\tilde w_\lambda,\nabla \tilde
  w_\lambda)\d x&\geq
  \liminf_{\lambda\to\infty}\frac{1-C(v_0,\nabla w_0)\e^{1/2}}{1+C(v_0,\nabla w_0)\e^{1/2}}\int_Q
  h_{\lambda_\lambda}(v_0,\nabla
  \tilde w_\lambda)\d x\\
  &\geq 2 \sqrt{1+|v_0|^2}\rho^0(S(v_0,\nabla w_0))\,.
\end{split}
\]
This proves equation \eqref{eq:2}.

\medskip

Recall $G(v,p)=2\rho^0(S(v,p))\sqrt{1+|v|^2}$, and
$G_\infty(v,p)=2|S(v,p)|_\infty\sqrt{1+|v|^2}$. 
Let $v_\lambda\to v$ weakly * in BV.  We have by Lemma \ref{lem:hladd} (iii) that
\[
h_\lambda(v_\lambda,\nabla v_\lambda)\geq G_\infty(v_\lambda,\nabla v_\lambda)
\]
for all $\lambda$.  By Theorem \ref{thm:MFsingularpart}, we have that for
$|D^sv|\ecke C_v$ almost every point $x_0\in\Omega$,
\[
\zeta_2(x_0)\geq G_\infty\left(v(x_0),\frac{\d Dv}{\d|Dv|}(x_0)\right)
\]
which proves \eqref{eq:4}, since $\frac{\d Dv}{\d|Dv|}(x_0)$ is rank one, and
hence
\[
2\rho^0\left(S\left(v(x_0),\frac{\d
      Dv}{\d|Dv|}(x_0)\right)\right)=2\left|S\left(v(x_0),\frac{\d
      Dv}{\d|Dv|}(x_0)\right)\right|_\infty\,.
\]
By Lemma \ref{lem:MFvariant}, we have in a similar fashion for $|D^sv|\ecke J_v$
almost every $x_0\in \Omega$,
\[
\zeta_3(x_0)\geq \tilde K_{G_\infty}(v^+(x_0),v^-(x_0),\nu(x_0))\,.
\]
By Lemma \ref{lem:Kcalc}, it follows
\[
\zeta_3(x_0)\geq 2\sqrt{1+|\nu^\bot\cdot v(x_0)|^2}\arccos \NN(v^+(x_0))\cdot
\NN(v^-(x_0))\,.
\]
This proves \eqref{eq:3} and completes the proof of the lower bound.
\end{proof}

\subsection{Upper bound}
For the proof of the upper bound, we will need a modification of the well known
result in the calculus of variations that states that the relaxation of integral
functionals with suitable integrands that depend on $x,u,\nabla u$ is obtained
by the quasiconvexification of the integrand with respect to the gradient
variable. Here, we will need the analogous result for integrands that depend on
$\nabla u, \nabla^2 u$. 

\begin{proposition}
\label{prop:relax2}
  Let $1\leq p<\infty$, and let $f:\R^{m\times 2}\times\R^{m\times 2\times 2}\to\R$
  such that 
  \begin{equation}\label{eq:6}
\left.\begin{split}
0\leq &f(v,\xi)\leq C (1+|\xi|^p)\\
|f(v,\xi)-f(\tilde v,\xi)|\leq &C|v-\tilde v|\max (f(v,\xi),f(\tilde v,\xi)) \\
|Q_2f(v,\xi)-Q_2f(\tilde v,\xi)|\leq &C|v-\tilde v| \max (Q_2f(v,\xi),Q_2f(\tilde v,\xi))
\end{split}\right\}
\,\forall
v,\tilde v\in\R^{2\times n}, \xi\in
\R^{m\times 2\times 2}.
\end{equation}
Furthermore, let $\Omega\subset \R^2$ be open and bounded with smooth boundary, $u\in W^{2,p}(\Omega)$  and $\delta>0$. Then there
exists $w\in W^{2,p}(\Omega;\R^m)$ with
\[
\begin{split}
  \|u-w\|_{W^{1,p}(\Omega;\R^m)}&<\delta\\
  \int_\Omega f(\nabla w,\nabla^2 w)\d x&< \int_\Omega Q_2f(\nabla u,\nabla^2 u)\d x+\delta\,.
\end{split}
\]
\end{proposition}

For the proof of the proposition, we are going to use 
\begin{lemma}[Lemma A.3 in \cite{olbermann2017michell}]
\label{lem:approx}
Let $\Omega\subset\R^2$ be open and bounded with smooth boundary, and let $p\in
[1,\infty)$.  Furthermore let $u\in C^3(\Omega)$ and $\delta>0$. Then there exists $w\in 
  W^{2,\infty}(\Omega)$ and $\Omega_w\subset\Omega$ such that $\Omega_w$ is
  the union of mutually disjoint closed cubes, $w$ is
  piecewise a polynomial of degree $2$ on $\Omega_w$, and furthermore
\[
\begin{split}
  \|u-w\|_{W^{2,p}(\Omega)}&<\delta\,,\\
\|w\|_{W^{2,\infty}}&\lesssim \|u\|_{W^{2,\infty}}\\
\int_{\Omega\setminus\Omega_w}(1+|\nabla^2u|^p+|\nabla^2w|^p)\d x&<\delta\,.
\end{split}
\]
\end{lemma}

\begin{remark}
We remark that the inequality $\|w\|_{W^{2,\infty}}\lesssim \|u\|_{W^{2,\infty}}$ does not appear in the statement of Lemma A.3 in \cite{olbermann2017michell}, but is clear from the proof there. Also, the assumptions on the domain $\Omega$ that we make here are sufficient; the assumption of simple connectedness made in \cite{olbermann2017michell} does not play a role in the proof of that lemma. Finally, we note in passing that neither there nor here the assumption that the dimension of the domain is $n=2$ is of any relevance; the respective statements hold true for general $n$ just as well.
\end{remark}

\begin{proof}[Proof of Proposition \ref{prop:relax2}]
First we recall the well known fact that rank-one convex functions are locally
Lipschitz continuous (see e.g. \cite{MR2361288}). 
This holds true in particular
for $Q_2f(v,\cdot)$ for any $v$, and hence the assumption \eqref{eq:6}
implies that $Q_2f$ is locally Lipschitz continuous in both arguments.
More precisely, with the assumed
growth properties for $f$, we  have $ Q_2f(v,\xi)\leq C(1+|\xi|^p)$ and
hence 
\begin{equation}
|Q_2(v,\xi)-Q_2(v,\tilde \xi)|\leq C|\xi-\tilde \xi|\left(1+|\xi|^{p-1}+|\tilde
  \xi|^{p-1}\right) \quad\forall \xi,\tilde \xi\in \R^{m\times n^2}
\label{eq:9}
\end{equation}
where $C$ is some constant that is independent of $v,\xi,\tilde \xi$ (see
Proposition 2.32 in 
\cite{MR2361288}).

We set $u_\e:=\eta_\e* u$ and claim that
\begin{equation}
\label{eq:7}
\lim_{\e\to 0}\int_\Omega Q_2f(\nabla u_\e,\nabla^2 u_\e)\d x=\int_\Omega
Q_2f(\nabla u,\nabla^2 u)\d x\,.
\end{equation}
Indeed,
we have that  $u_\e\to u$ in $W^{2,p}$, and hence by \eqref{eq:9} and the
assumption \eqref{eq:6}, we have
\begin{equation}
\label{eq:42}
\begin{split}
  \int_\Omega &|Q_2f(\nabla u_\e,\nabla^2 u_\e)- Q_2f(\nabla u,\nabla^2 u)|\d
  x\\
&\leq \int_\Omega |Q_2f(\nabla u_\e,\nabla^2 u_\e)-
  Q_2f(\nabla u_\e,\nabla^2 u)|
   + |Q_2f(\nabla u_\e,\nabla^2 u)-
  Q_2f(\nabla u,\nabla^2 u)|\d x\\
  &\leq C\int_\Omega |\nabla^2 u_\e-\nabla^2 u| (1+|\nabla^2
  u|^{p-1}+|\nabla^2 u_\e|^{p-1})+\max(|\nabla u_\e-\nabla u|,1) (1+|\nabla^2 u|^p)
  \d x
\end{split}
\end{equation}
For $\e\to 0$, the integral on the right hand side  converges to 0, proving the claim \eqref{eq:7}.

Let $\Delta>0$. We choose $u_\e$ such that $\|u-u_\e\|_{W^{2,p}}<\Delta$.
By Lemma \ref{lem:approx}, there exists $w_\e\in W^{2,\infty}$ and a union of
disjoint closed cubes $\Omega_w\subset \Omega$ such that each component of  $w_\e$ is a polynomial
of degree 2 on each of the cubes, and 
\[
\begin{split}
  \|w_\e-u_\e\|_{W^{2,p}}&<\Delta\\
  |\Omega\setminus\Omega_w|(1+\|w_\e\|_{W^{2,\infty}}^p)&<\Delta\,.
\end{split}
\]
By the same kind of estimate as in \eqref{eq:42}, we obtain that additionally, we may choose $w_\e$, $\Omega_w$
such that
\begin{equation}
\int_\Omega Q_2f(\nabla w_\e,\nabla^2 w_\e)\d x<\int_\Omega Q_2f(\nabla
u_\e,\nabla^2 u_\e)\d x+\Delta\,.\label{eq:8}
\end{equation}
Moreover, we may choose the cubes to be so small that on each cube $\tilde
Q$ with center $x_0$ in the collection,
\[
\sup_{x\in\tilde Q}|\nabla w_\e(x)-\nabla w_\e(x_0)|<\Delta\,.
\]

Let $\tilde Q$ be a cube where the components of $w_\e$ are quadratic polynomials, with midpoint
$x_0$ and sidelength $r$. Choose $\xi\in W_0^{2,\infty}(\tilde Q)$ such that
\[
\int_{\tilde Q} f(\nabla
w_\e(x_0),\nabla^2 w_\e(x_0)+\nabla^2\xi)\d x\leq \mathrm{vol}(\tilde Q) Q_2f(\nabla w_\e(x_0),\nabla^2 w_\e(x_0)) +\frac\Delta N\,,
\]
where $N$ is the total number of cubes. 
Let us write $\tilde \xi(x)=\xi(x_0+rx)$ for $x\in [-1/2,1/2]^2$, and define $\tilde \xi$ on $\R^2$ by
1-periodic extension.
For $M\in\N$, let
\[
\xi_M(x)=M^{-2} \tilde \xi\left(M(x-x_0)\right)\,.
\]
We have that  $\|\xi\|_{W^{1,\infty}}\to 0$ for $M\to \infty$, $\|\nabla^2\xi\|_{L^\infty}=\|\nabla^2\xi_M\|_{L^\infty}$,  and 
\[
\int_{\tilde Q} f(\nabla
w_\e(x_0),\nabla^2 w_\e(x_0)+\nabla^2\xi)\d x=\int_{\tilde Q} f(\nabla
w_\e(x_0),\nabla^2 w_\e(x_0)+\nabla^2\xi_M)\d x\,.
\]
We choose  $M$
so large that $\|\nabla \xi_M\|_{L^\infty}<\Delta$.
This implies 
\begin{equation}
  \label{eq:45}
\|\nabla w_\e+\nabla \xi_M-\nabla w_\e(x_0)\|_{L^\infty(\tilde Q)}<2\Delta\,.
\end{equation}
Using the local Lipschitz continuity in the first argument of $f$ as assumed in \eqref{eq:6},
\[
\int_{\tilde Q} f(\nabla
w_\e+\nabla \xi_M,\nabla^2 w_\e(x_0)+\nabla^2\xi_M)\d x<
\frac{1+C\Delta}{1-C\Delta}\int_{\tilde Q} f(\nabla
w_\e(x_0),\nabla^2 w_\e(x_0)+\nabla^2\xi_M)\d x\,.
\]
We repeat the same for all cubes $\tilde Q$ in $\Omega_w$, obtaining a corrector
function $\xi_{\tilde Q}\in W^{2,\infty}_0(\tilde Q)$ in each of them.
We set $\bar w= w_\e+\sum_{\tilde Q} \xi_{\tilde Q}$. Denoting by $x_{\tilde Q}$
the center of the cube $\tilde Q$, we have
\[
\begin{split}
  \int_{\Omega} f(\nabla\bar w ,\nabla^2 \bar w)\d x&\leq
\sum_{\tilde Q}\int_{\tilde Q}f(\nabla \bar w,\nabla^2\bar w)\d
x
+\int_{\Omega\setminus\Omega_w}f(\nabla \bar w,\nabla^2\bar w)\d x\\
&\leq \sum_{\tilde Q}\frac{1+C\Delta}{1-C\Delta}\int_{\tilde Q}f(\nabla 
w_\e(x_{\tilde Q}),\nabla^2\bar w)\d x+C \int_{\Omega\setminus\Omega_w}(1+|\nabla^2\bar
w|^p)\d x\\
&\leq  \sum_{\tilde Q}\frac{1+C\Delta}{1-C\Delta}\left(\int_{\tilde Q}Q_2f(\nabla 
w_\e
(x_{\tilde Q}),\nabla^2 w_\e)\d x+\frac{\Delta}{N}\right)+C\Delta\\
&\leq \left(\frac{1+C\Delta}{1-C\Delta}\right)^2 \int_\Omega Q_2f(\nabla
w_\e,\nabla^2 w_\e)\d x+C\Delta\,.
\end{split}
\]
Here we used again \eqref{eq:45} in combination with the assumption \eqref{eq:6} to obtain the
last inequality.
By $\|u-w_\e\|_{W^{2,p}}< 2\Delta$ and \eqref{eq:7}, this last estimate proves the claim
by choosing $\Delta$ small enough.
\end{proof}

\begin{proof}[Proof of Theorem \ref{thm:main} (iii)]
Just as for the lower  bound, we will use the blow-up method for the proof of
the upper bound, in combination with a suitable mollification.
We assume that we are given a sequence $\lambda_j\to \infty$, and we will prove
that for any subsequence, there exists a further subsequence fulfilling the
upper bound. We omit the index $j$ from our notation and write $\lambda\to
  \infty$ for $j\to \infty$.

\medskip

{\em Step 1}:
Let $\hat \Omega$ be
some neighborhood of $\overline\Omega$ in $\R^2$.
By Theorem \ref{thm:traceop},  there exists a function $v\in
W^{2,1}(\hat\Omega\setminus \overline \Omega)$  such that the traces $\gamma_0(v),\gamma_1(v)$
on $\partial\Omega$
with respect to $\hat\Omega\setminus\overline\Omega$
are identical with $\gamma_0(u)$, $\gamma_1(u)$ (up to the appropriate sign
change for $\gamma_1$). By Theorem 3.84 in \cite{MR1857292}, the function
$w:=u\chi_\Omega+v\chi_{\hat\Omega\setminus\Omega}$ is an element of $BH(\hat\Omega)$, with
$|D\nabla w|(\partial\Omega)=0$.

\medskip

For $\e>0$ small enough, we have $\{x\in \R^2:\dist(x,\Omega)<\e\}\subset \hat\Omega$,
and we may set $u_\e(x)=(\eta_\e * w)(x)$ for $x\in \Omega$. Here
$\e=\e(\lambda,u)$ is chosen such that 
\begin{equation}
    \|S(\nabla u_\e,\nabla^2 u_\e)\|_{L^\infty}\leq
    \|\nabla^2 u_\e\|_{L^\infty}< \sqrt{\lambda}/2\,,\label{eq:36}
\end{equation}
and $\e(\lambda,u)\to 0$ as $\lambda\to
\infty$.
Such a choice of $\e$ is possible by the standard estimate
\[
  \|\nabla^2 u_\e\|_{L^\infty(\Omega)}\leq C\e^{-2}|D\nabla w|(\hat\Omega)\,
\]
where $C=\sup |\eta|$; the first inequality in \eqref{eq:36} being
valid by Lemma \ref{lem:basicSestimate}.
Writing $u_\lambda=u_{\e(\lambda,u)}$, $v_\lambda=\nabla u_\lambda$, $v=\nabla u$ and recalling the definition
\eqref{eq:55} of
$h_\lambda$,  we have that
\begin{equation}
\begin{split}
  h_\lambda(v_\lambda,\nabla v_\lambda)&\leq 2\sqrt{1+|v_\lambda|^2}\rho^0(S(v_\lambda,\nabla v_\lambda))\\
  &\leq 2 |\nabla v_\lambda|\,.
\end{split}\label{eq:16}
\end{equation}
Hence we have that after
passing to a subsequence, there exists a measure $\mu$ such that
\begin{equation}
\label{eq:56}
\left.\begin{split}h_\lambda(v_\lambda,\nabla v_\lambda)\L^2\ecke \Omega
  &\to\mu\\
\nabla v_\lambda\L^2&\to Dv\\
|\nabla v_\lambda|\L^2&\to |Dv|\end{split}\right\}
\text{ weakly * in the sense of measures}
\end{equation}
Additionally, it follows from  $|D\nabla w|(\partial \Omega)=0$ and
\eqref{eq:16} that
${\mu(\partial\Omega)=0}$, and hence \[
  \lim_{\lambda\to\infty}\int_\Omega
  h_\lambda(v_\lambda,\nabla v_\lambda)\d x= \mu(\Omega)\,.
  \]

\medskip

{\em Step 2.} 
For every continuous non-negative function $\varphi\in C_0(\Omega)$, we have
\[
\begin{split}
  \int_\Omega\varphi\,\d\mu&=
  \lim_{\lambda\to\infty}\int_\Omega \varphi\,h_\lambda(
  v_\lambda,\nabla
  v_\lambda)\d x\\
&\leq  2\lim_{\lambda\to\infty}\int_\Omega\varphi |\nabla v_\lambda|\,\d x\\
&= 2\int_\Omega\varphi \,\d|Dv|\,.
\end{split}
\]
Hence, $\mu$ is absolutely continuous with respect to $|Dv|$, and
the measure $\mu$ can be decomposed into mutually singular measures,
\[
\mu=\zeta_1\L^2+\zeta_2 |D^sv|\ecke C_{v}+\zeta_3 \H^1 \ecke
J_{v}\,\,.
\]
We will prove
\begin{align}
  \zeta_1(x)\leq &2\rho^0\left(S\left(v(x),\nabla v(x)\right)\right)\sqrt{1+|v|^2} \quad
  \text{ for }\L^2-\text{a.e.}\, x\in \Omega \label{eq:17}\\
  \zeta_2(x)\leq &2\rho^0\left(S\left(v(x),\frac{\d Dv}{\d
        |Dv|}(x)\right)\right)\sqrt{1+|v|^2} \quad \text{ for
  }|D^sv|-\text{a.e.}\, x\in C_{v}\label{eq:19}\\
\zeta_3(x)\leq &2\arccos\left(\NN(v^+)\cdot \NN(
    v^-)\right)\sqrt{1+|\nu^\bot\cdot v|^2} \label{eq:18}\quad
  \text{ for }\H^1-\text{a.e.}\, x\in J_{v}\,.
\end{align}
Once we have proved these inequalities, we have proved
\[
\limsup_{\lambda\to \infty}\int_\Omega h_\lambda(v_\lambda,\nabla v_\lambda)\d x\leq\F(u)\,.
\]
The upper bound then follows by $Q_2 f_\lambda=h_\lambda$ and Proposition
\ref{prop:relax2}. Indeed, the assumptions of the proposition (with $p=2$) are
fulfilled for $f_\lambda$ by
Lemma \ref{lem:hbounds}.

\medskip

{\em Step 3.} To prove \eqref{eq:17},
we use that at $\L^2$-almost every $x_0$,  $v=\nabla u$ is approximately differentiable, i.e., 
\begin{equation}
\lim_{r\to 0}\frac{1}{r}\fint_{Q(x_0,r)}|v(x)-v(x_0)-\nabla v(x_0)\cdot(x-x_0)|\d
  x=0\,.\label{eq:47}
\end{equation}
In particular, we may assume that $x_0$ is a Lebesgue point of $v$,
\[
  \lim_{r\to 0}\fint |v(x)-v(x_0)|\d x=0\,.
  \]
We define
\[
  \begin{split}
    v^{(\rho)}(x)&=\rho^{-1}(v(x_0+\rho x)-v(x_0))\\
    V(x)&=\nabla v(x_0)\cdot(x-x_0)
  \end{split}
\]
and have by \eqref{eq:47} that
\[
  v^{(\rho)}\to V \text{ in }L^1(Q(0,2))
\]
 and by Theorem \ref{thm:RN} that
\[    
\left.\begin{split}
      Dv^{(\rho)}&\to \nabla v(x_0)\L^2\\
      |Dv^{(\rho)}|&\to |\nabla v(x_0)|\L^2\end{split}\right\}
  \text{ weakly * in the sense of measures.}
\]
Now we choose $\rho(\lambda)$ with $\e(\lambda,u)<\rho(\lambda)/2$.
In the sequel, we write $\rho\equiv \rho(\lambda)$. We set
\[
  v^{(\rho)}_\lambda(x)=\rho^{-1}(v_\lambda(x_0+\rho x)-v(x_0))
\]
and have
\[
  v^{(\rho)}_\lambda\to V \text{ in }L^1(Q(0,2))
\]
 and 
  \begin{equation}
\left.\begin{split}
      \nabla v^{(\rho)}_\lambda \L^2&\to \nabla v(x_0)\L^2\\
      |\nabla v^{(\rho)}_\lambda|\L^2&\to |\nabla v(x_0)|\L^2\end{split}\right\}
  \text{ weakly * in the sense of measures.}\label{eq:57}
\end{equation}
Moreover,
\begin{equation}
  \label{eq:58}
  \begin{split}
    \sup_{x\in Q(x_0,\rho)} |v_\lambda(x)-v(x_0)|
    &\leq \sup_{ Q(x_0,\rho)} |\eta_{\e(\lambda,u)}*(v-v(x_0))|\\
    &\leq \|v_\lambda-v(x_0)\|_{L^1(Q(x_0,2\rho))}\\
      &\to 0\text{ for }\lambda\to \infty\,.
    \end{split}
\end{equation}

By the Radon-Nikodym Theorem,
\[
  \zeta_1(x_0)=\lim_{\lambda\to \infty} \fint_{Q(x_0,\rho)}
  h_\lambda(v_\lambda,\nabla v_\lambda)\d x\,.
\]

We recall that $G:\R^2\times \Rsym\to \R$ is defined by
\[G(\xi,p)=2\rho^0(S(\xi,p))\sqrt{1+|\xi|^2}\,.
\]
By Lemma \ref{lem:hbounds} we have that
  \begin{equation}
  |G(\xi,p)-G(\tilde \xi,p)|\lesssim |\xi-\tilde \xi| \max
  (G(\xi,p),G(\tilde\xi,p))\,.\label{eq:59}
  \end{equation}

By combining  \eqref{eq:36} with the definition \eqref{eq:55} of $h_\lambda$,
 we have that $h_\lambda( v_\lambda,\nabla v_\lambda)\leq G(
 v_\lambda,\nabla v_\lambda)$, and hence

\[
  \zeta_1(x_0)\leq \limsup_{\lambda\to \infty} \fint_{Q(x_0,\rho)}
  G(v_\lambda,\nabla v_\lambda)\d x\,.
\]

By \eqref{eq:58} and \eqref{eq:59}, we obtain

\[
  \begin{split}
    \zeta_1(x_0)&\leq \limsup_{\lambda\to\infty}\fint_{Q(x_0,\rho)}
    G(v(x_0),\nabla v_\lambda)\d  x \\
    &=\limsup_{\lambda\to\infty}\int_{Q(0,1)} G(v(x_0),\nabla
    v_\lambda^{(\rho)})\d x\,.
  \end{split}
  \]
  By Theorem \ref{thm:delladio} and \eqref{eq:57}, we finally get
  \[
    \begin{split}
      \zeta_1(x_0)&\leq \int_{Q(0,1)} G(v(x_0),\nabla v(x_0))\d x\\
      &=G(v(x_0),\nabla v(x_0))\,,
    \end{split}
    \]
which proves \eqref{eq:17}.


\medskip

{\em Step 4.} 
For $|D^sv|\ecke C_{v}$-almost every $x_0$, we have that by Alberti's rank one Theorem \cite{alberti1993rank} and Theorem \ref{thm:BVblow}
the following holds true:  There exists a sequence
$\rho_l\downarrow 0$ and 
a  monotone function $\psi\in BV(-1/2,1/2)$ 
such that the rescaled functions 
\[
v^{(\rho_l)}(x):=\frac{\rho_l}{|Dv|(Q_\xi(x_0,\rho_l))}\left(
  v(x_0+\rho_l x)-\fint_{Q_\xi(x_0,\rho_l)}v (x')\d x'\right)
\]
converge for $l\to \infty$ in $L^1(Q_\xi;\R^2)$ to
the function 
\begin{equation*}
\Psi:x\mapsto \xi \psi(x\cdot \xi)\,,
\end{equation*}
where $\xi\in S^1$ fulfills $\frac{\d(Dv)}{\d|Dv|}(x_0)=\xi\otimes
\xi$.
Also, from Theorem \ref{thm:RN}, we have that
  \begin{equation}
    \label{eq:61}
  \left.\begin{split}  Dv^{(\rho_l)}&\to D\Psi \\ |Dv^{(\rho_l)}|&\to |D\Psi|\end{split}\right\}\text{ weakly * in the sense of measures.}
\end{equation}

From now on, in order to alleviate the notation, we are going to omit the index
$l$ from $\rho_l$, and we write $\lim_{\rho\to 0}$ for $\lim_{l\to\infty}$.
As a consequence of the convergence of $v^{(\rho)}\to\Psi$ in $L^1$,  we have in particular that $x_0$ is a Lebesgue point of $v$,
\[
  \lim_{\rho\to 0}\fint_{Q(x_0,\rho)}|v-v(x_0)|\d x=0\,.
  \]

By $|Dv^{(\rho)}|(Q_\xi)=1$ and the lower semicontinuity of total variation, we have
that $|D\Psi|(Q_\xi)=|D\psi|(-1/2,1/2)\leq 1$.

\medskip


Now we choose  $\rho(\lambda)\to 0$ such that $\mu(\partial
Q(x_0,\rho(\lambda)))=0$, $\e/\rho\to 0$, and 
\begin{equation*}
\zeta_2(x_0)=
\lim_{\lambda\to 0}\frac{1}{|D^sv
  |(Q_\xi(x_0,\rho(\lambda)))}\int_{Q_\xi(x_0,\rho(\lambda))}h_\lambda(v_\lambda,\nabla v_\lambda)\d x\,.
\end{equation*}

Writing $\rho\equiv\rho(\lambda)$, we note that
\begin{equation}
\begin{split}
  (v^{(\rho)})_{\e/\rho}(x)&=\frac{\rho}{|Dv|(Q_\xi(x_0,\rho))}v_\lambda(x_0+\rho x)\\
  \left(\nabla (v^{(\rho)})_{\e/\rho}\right)(x)
  &=\frac{\rho^2}{|Dv|(Q_\xi(x_0,\rho))} \left(\nabla v_\lambda\right)(x_0+\rho x)\,.
\end{split}\label{eq:44}
\end{equation}

  As in the previous step, using the fact that $x_0$ is a Lebesgue point for $v$, we obtain
    \begin{equation}
      \label{eq:60}
    \sup_{x\in Q(x_0,\rho)}|v_\lambda(x)-v(x_0)|\to 0 \text{ as }\lambda\to 0\,.
  \end{equation}




Using $h_\lambda\leq G$ (see again step 3),
we have that

\begin{equation}
\label{eq:50}
  \begin{split}
    \limsup_{\lambda\to\infty}&\frac{1}{|D^sv
      |(Q_\xi(x_0,\rho))}\int_{Q_\xi(x_0,\rho)} h_\lambda\left(v_\lambda
      ,\nabla v_\lambda
      \right)\d x\\
&\leq \limsup_{\lambda\to\infty}\frac{1}{|D^s
      v|(Q_\xi(x_0,\rho))}\int_{Q_\xi(x_0,\rho)} G\left(v_\lambda
      ,\nabla v_\lambda
      \right)\d x\,.
\end{split}
\end{equation}
By a change of variables, we obtain
\begin{equation*}
  \label{eq:37}
  \int_{Q_\xi(x_0,\rho)} G\left(v_\lambda,\nabla v_\lambda
      \right)\d x=\rho^2\int_{Q_\xi}G\left(v_\lambda
      (x_0+\rho y),\nabla v_\lambda
      (x_0+\rho y)\right)\d y\,.
  \end{equation*}
Combining this with \eqref{eq:44} and \eqref{eq:50}, we obtain
\begin{equation*}
\label{eq:62}
  \begin{split}
    \limsup_{\lambda\to\infty}&\frac{1}{|D^sv
      |(Q_\xi(x_0,\rho))}\int_{Q_\xi(x_0,\rho)} h_\lambda\left(v_\lambda
      ,\nabla v_\lambda
      \right)\d x\\
      &\leq \limsup_{\lambda\to\infty}\int_{Q_\xi} G\left(v_\lambda(x_0+\rho y)
      ,\nabla \left(v^{(\rho)}\right)_{\e/\rho}
      \right)\d y\,.
\end{split}
\end{equation*}
  
By \eqref{eq:60}, this yields
\begin{equation}
\label{eq:51}
  \begin{split}
    \limsup_{\lambda\to\infty}&\frac{1}{|D^sv
      |(Q_\xi(x_0,\rho))}\int_{Q_\xi(x_0,\rho)} h_\lambda\left(v_\lambda
      ,\nabla v_\lambda
       \right)\d x\\
&\leq \limsup_{\lambda\to\infty}\int_{Q_\xi} G\left(v(x_0)
      ,\nabla \left(v^{(\rho)}\right)_{\e/\rho}
\right)\d x\,.
\end{split}
\end{equation}
By \eqref{eq:61} and $\e/\rho\to 0$, we have that

\begin{equation*}
    \label{eq:62}
  \left.\begin{split}  \nabla\left(v^{(\rho)}\right)_{\e/\rho}\L^2&\to D\Psi \\ \left|\nabla \left(v^{(\rho)}\right)_{\e/\rho}\right|\L^2&\to |D\Psi|\end{split}\right\}\text{ weakly * in the sense of measures.}
\end{equation*}  
By Theorem \ref{thm:delladio}, it follows that
\[
  \begin{split}
 \limsup_{\lambda\to\infty}\int_{Q_\xi} G\left(v(x_0)
      ,\nabla \left(v^{(\rho)}\right)_{\e/\rho}
\right)\d x &= \int_{Q_\xi}G\left(v(x_0)
      ,\xi\otimes\xi
    \right)\d|D\Psi| \\
    &\leq G\left(v(x_0)
      ,\xi\otimes\xi
    \right)\,,
  \end{split}
\]

which (together with \eqref{eq:51}) proves
 \eqref{eq:19}.

\medskip

{\em Step 5.} For $\H^1\ecke J_{\nabla u}$-almost every $x_0$, we have that by Theorem \ref{thm:BVblow}
the following holds true:
The rescaled functions 
\[
v^{(\rho)}(x)=v
  (x_0+\rho x)-\fint_{Q_\nu(x_0,\rho)}v(x')\d x'
\]
converge in $L^1(2Q_\nu;\R^2)$ to
the function 
\begin{equation*}
\Psi:x\mapsto \begin{cases} v^+(x_0)&\text{ if } x\cdot\nu >0\\
v^-(x_0) &\text{ if } x\cdot\nu \leq 0\,.\end{cases}
\end{equation*}

Additionally we have, for every $\beta>0$, 
\[
  \lim_{\rho\to 0}\int_{(1+\beta)Q_\nu}|\nabla v^{(\rho)}|\d x=|D\Psi|((1+\beta)Q_\nu)\,.
  \]

Now we choose $\rho(\lambda)$  such that $\rho\to 0$,
and $\e/\rho\to 0$  as $\lambda\to \infty$. 
Then we   may write, again using $h_\lambda\leq G$,
\begin{equation}
\label{eq:26}  
\begin{split}
  \lim_{\lambda\to 0}\frac{1}{\rho}\int_{Q_\nu(x_0,\rho)} h_\lambda(v_\lambda
  ,\nabla v_\lambda)\d x
  &\leq \liminf_{\lambda\to 0}\frac{1}{\rho}\int_{Q_\nu(x_0,\rho)} G(v_\lambda
  ,\nabla v_\lambda)\d x\\
  &\leq \liminf_{\lambda\to 0}\int_{Q_\nu} \rho G(v_\lambda
   (x_0+\rho x),\nabla v_\lambda(x_0+\rho x))\d x\\
  &=\liminf_{\lambda\to 0}\int_{Q_\nu} G\left(\eta_{\e/\rho}*v^{(\rho)},
      \nabla\eta_{\e/\rho}*v^{(\rho)}\right)\d x\,.
  \end{split}
\end{equation}

Since $\e/\rho\to 0$, we have that
\[
  \begin{split}
  \eta_{\e/\rho}*v^{(\rho)}&\to \Psi \quad\text{ in }L^1(Q_\nu)\\
  \lim_{\lambda\to \infty}\int_{Q_\nu}|\nabla \eta_{\e/\rho}*v^{(\rho)}|\d x&=|D\Psi|(Q_\nu)\,.
\end{split}
\]
By Proposition \ref{prop:up_jump_blowup}, which will be proved in Section \ref{sec:proof-jump} below, this implies 
\[
  \begin{split}
    \limsup_{\lambda\to\infty}\int_{Q_\nu}& G(\eta_{\e/\rho}*v^{(\rho)}, \nabla \eta_{\e/\rho}*v^{(\rho)})\d x\\
    &\leq 2\sqrt{1+|\nu^\bot\cdot v(x_0)|^2} \arccos(\NN(v^+(x_0))\cdot\NN(v^-(x_0)))\,,
\end{split}
\]
and  combining the latter with \eqref{eq:26} proves \eqref{eq:18}, since the left hand side of \eqref{eq:26} is just $\zeta_3(x_0)$.
\end{proof}

\section{Proof of the upper bound for the blow-up of the jump part}
\label{sec:proof-jump}

We recall that $G(v,\xi)=2\rho^0(S(v,\xi))\sqrt{1+|v|^2}$. In the present section, we write $I=[-1/2,1/2]$.

\begin{lemma}
\label{lem:cap_estim}
  Let $v_j\in C^1( Q)$ such that $v_j\to 0$ in $W^{1,1}(Q)$. Furthermore, let $P:\R^2\to\R$ be the projection $P(x_1,x_2)=x_1$, $\Delta>0$ and for $j\in \N$,
  \[
    A_j:=\{x\in Q:|v_j(x)|>\Delta\}\,.
  \]
  Then $\L^1(P(A_j))\to 0$ as $j\to \infty$.
\end{lemma}
\begin{proof}
  By the continuity of the trace operator $W^{1,1}(Q)\to L^1(\partial Q)$, we have that
  \[
    v_j(\cdot,-1/2)\to 0,\quad v_j(\cdot,1/2)\to 0 \quad\text{ in }L^1(I)\,.
  \]
  Setting
  \[
    \begin{split}
    J^1_j&:=\left\{t\in I: |v_j(t,-1/2)|+|v_j(t,1/2)|>\frac{\Delta}{2}\right\}\,,\\
    J^2_j&:=\left\{t\in I:\int_I|\partial_{x_2}v_j(t,x_2)|\d x_2>\frac{\Delta}{2}\right\}\,,
  \end{split}
  \]
  we have that $P(A_j)\subset J^1_j\cup J^2_j$ and $\L^1(J_j^1)+\L^1(J_j^2)\to 0$. This proves the claim.  
\end{proof}

\begin{lemma}
\label{lem:G1d_estim}  Let $a_1\in\R$ and  $w\in C^1(I)$. Then
  \[
    \begin{split}
      \int_{-1/2}^{1/2}G&\left( \left(\begin{array}{c}a_1\\w(t)\end{array}\right),\left(\begin{array}{cc} 0& 0\\0&w'(t)\end{array}\right)\right)\d t \\
      &\leq 2 \sqrt{1+a_1^2}\left(\arccos \NN\left(\begin{array}{c}a_1\\w(-1/2)\end{array}\right)\cdot \NN\left(\begin{array}{c}a_1\\w(+1/2)\end{array}\right)
      +2\int_{-1/2}^{1/2}|\min(0,w'(t))|\d t\right)\,.
  \end{split}
      \]
\end{lemma}

\begin{proof}
First we recall the geometric meaning of the integral: Let
  \[
    W(x_1,x_2)=\int_0^{x_2}w(t)\d t+a_1 x_1\,.
  \]
  Then we have that
  \[
    \begin{split}
    \int_{-1/2}^{1/2} G\left( \left(\begin{array}{c}a_1\\w(t)\end{array}\right),\left(\begin{array}{cc} 0& 0\\0&w'(t)\end{array}\right)\right)\d t
&    =\int_Q G(\nabla W,\nabla^2 W)\d x\\
    &= \int_{\gr W|_Q} 2\rho^0(S_{\gr W|_Q})\d\H^2\,.
  \end{split}
\]
As in the proof of Lemma \ref{lem:Kcalc}, we may rotate, cut and glue the graph $\gr W_Q$ to obtain the graph of the function
\[
  \tilde W:\left[-\frac{\sqrt{1+a_1^2}}{2},\frac{\sqrt{1+a_1^2}}{2}\right]\times[-1/2,1/2]\to \R\,,
\]
defined by
\[
  \tilde W(x_1,x_2)=\frac{1}{\sqrt{1+a_1^2}} \int_0^{x_2} w(t)\d t\,,
\]
without changing the ``energy'', i.e., $\tilde W$ satisfies
\[
  \int_{\gr W|_Q} 2\rho^0(S_{\gr W|_Q})\d\H^2=\int_{\gr \tilde W|_{\tilde Q}} 2\rho^0(S_{\gr \tilde W|_Q})\d\H^2\,,
  \]
  where we have written $\tilde Q=\left[-\frac{\sqrt{1+a_1^2}}{2},\frac{\sqrt{1+a_1^2}}{2}\right]\times[-1/2,1/2]$. Now $\tilde W$ is constant in $x_1$-direction, which implies that the normal vector is contained in the $x_2$-$x_3$ plane,
  \[\NN(\nabla\tilde W)\in\{(0,x_2,x_3)^T:x_2^2+x_3^2=1\}\,,\]
  which implies
  \[\rho^0(S_{\gr \tilde W})=\left(1+|\partial_{x_2}\tilde W|^2\right)^{-1/2}|\partial_{x_2}\NN(\nabla\tilde W)|\,.
  \]

  This leaves us, after the passage back to a one-dimensional setting, with the following  calculation:
  \[
      \int_{\gr \tilde W|_{\tilde Q}} 2\rho^0(S_{\gr \tilde W|_Q})\d\H^2 
     =
      2\sqrt{1+a_1^2}\int_{-1/2}^{1/2} \left|\partial_t \NN\left(\begin{array}{c}0\\w(t)/\sqrt{1+a_1^2}\end{array}\right)\right|\d t\,.
\]

We observe $\NN((0,w(t)/\sqrt{1+a_1^2})^T)=\NN((0,w(t))^T)$.  By giving $\{(0,x_2,x_3)^T:x_2^2+x_3^2=1\}\simeq S^1$ an orientation (the one corresponding to increasing $w(t)$), the total variation of the curve
$\NN((0,w(\cdot))^T)$ can be estimated from above by the distance of its endpoints on $S^1$, plus twice the integral of $|\partial_t\NN((0,w(t))^T)|$ over the region where $\partial_t\NN((0,w(t))^T)$ is anti-parallel to the orientation:

\[
\begin{split}
  \int_{-1/2}^{1/2} \left|\partial_t \NN\left(\begin{array}{c}0\\w(t)\end{array}\right)\right|\d t
    &\leq  \arccos \NN\left(\begin{array}{c}0\\w(-1/2)\end{array}\right)\cdot \NN\left(\begin{array}{c}0\\w(+1/2)\end{array}\right)\\
    &\quad+2\int_{\{t:w'(t)\leq 0\}}  \left|\partial_t \NN\left(\begin{array}{c}0\\w(t)\end{array}\right)\right|\d t\\
  &\leq     
 \arccos \NN\left(\begin{array}{c}a_1\\w(-1/2)\end{array}\right)\cdot \NN\left(\begin{array}{c}a_1\\w(+1/2)\end{array}\right)\\
&\quad +2       \int_{-1/2}^{1/2}|\min(0,w'(t))|\d t\,.
  \end{split}
\]
In the last inequality, we have used the fact that the angle between the normals at the endpoints does not change when applying a rotation, and $|\partial_t\NN(\tilde w(t))|\leq|\partial_t \tilde w(t)|$. This completes the proof of  the lemma.
\end{proof}

\begin{proposition}
  \label{prop:up_jump_blowup}
  Let $a\in \R^2$, $\nu\in S^1$, $t\in \R$, $b=a+t\nu$, and $\Psi:Q_\nu\to\R^2$ defined by
 
  \[
    \Psi(x)=\begin{cases}a&\text{ if } x\cdot \nu\leq 0\\b& \text{ if }x\cdot \nu>0\,.\end{cases}
  \]
  Assume that $v_j\in C^1(Q_\nu;\R^2)$ is a sequence such that
  \[
    \begin{split}
    v_j&\to\Psi \quad\text{ in } L^1(Q_\nu;\R^2)\\
    \int_{Q_\nu}|\nabla v_j|\d x &\to |D\Psi|(Q_\nu)=|t|\,.
  \end{split}
    \]
    Then
    \[
      \limsup_{j\to\infty} \int_{Q_\nu} G(v_j,\nabla v_j)\d x\leq 2\sqrt{ 1+|a\cdot\nu^\bot|^2}\arccos \NN(b)\cdot \NN(a)\,.
      \]
\end{proposition}

\begin{proof}
  We may assume without loss of generality that $t\geq 0$, $\nu=e_2$,  which implies $b=(a_1,a_2+t)^T$. Let $P:\R^2\to\R$ be defined by  $P(x_1,x_2)=x_1$.

\medskip
  
Fix $\Delta>0$. For $j\in\N$, we will split $Q\equiv Q_\nu$ into a ``good'' and a ``bad'' set, and estimate the energy on these sets separately. We write $v_j=((v_j)_1,(v_j)_2)^T$ and define the bad set by setting
\[ 
  \begin{split}
  \tilde A_j&:=\{x\in Q:  |(v_j)_1( x)-a_1|>\Delta\}\,,\\
  A_j&:=P^{-1}(P(\tilde A_j))
\end{split}
\]
The assumptions of the present proposition imply in particular that
\[
  \liminf_{j\to\infty}\int_Q|\nabla (v_j)_2|\d x\geq |D\Psi_2|(Q_\nu)=t\,,
\]
and hence
  \begin{equation}
    \label{eq:27}
    \begin{split}
      (v_j)_1-a_1&\to 0 \quad\text{ in }W^{1,1}(Q)\\
      \partial_{x_1}(v_j)_2&\to 0 \quad\text{ in }L^{1}(Q)\,.
  \end{split}
\end{equation}


  Next we claim that there exists a sequence $\beta_j\downarrow 0$ such that for any (measurable) $J\subset [-1/2,1/2]$ and any $j\in\N$, we have that 
    \begin{equation}
\label{eq:54}    \left| \int_J |v_j(x_1,1/2)-v_j(x_1,-1/2)|\d x_1- t\L^1(J)\right|\leq \beta_j\,.
  \end{equation}
  To prove this claim, we note that we thanks to  the continuity of the trace operator $W^{1,1}(Q)\to L^1(\partial Q)$, we have
  \[
    \left.\begin{split}
    v_j(\cdot,-1/2)&\to a \\
        v_j(\cdot,1/2)&\to b 
      \end{split}\right\}\quad\text{ in }L^1(I)\,.
  \]
  Hence, with
  \[
    \beta_j:= \int_{-1/2}^{+1/2} |v_j(x_1,1/2)-v_j(x_1,-1/2)-(b-a)|\d x_1\,,
  \]
  we obtain the claim \eqref{eq:54} by the triangle inequality.

  Additionally, by possibly increasing $\beta_j$, but still retaining the property $\beta_j\downarrow 0$, we may achieve
  \begin{equation}
    \label{eq:64}
    \left|\int_Q|\nabla v_j|\d x-t\right|\leq\beta_j\,.
  \end{equation}
    
  We set $I_j=P(A_j)$. By \eqref{eq:27} and Lemma \ref{lem:cap_estim}, we have that $\L^1(I_j)\to 0$.
  We are now in position to estimate the contribution of the bad sets $A_j$:
    \begin{equation}
    \label{eq:65}\begin{split}
    \limsup_{j\to\infty} \int_{A_j}|\nabla v_j|\d x& =\limsup_{j\to\infty}
    \left|\int_{A_j}|\nabla v_j|\d x-t\L^1(I_j)\right|\\
    &\leq \limsup_{j\to\infty} \left|\int_{Q}|\nabla v_j|\d x-t\right|
    +\left|\int_{Q\setminus A_j}|\nabla v_j|\d x-t\L^1(I\setminus I_j)\right|\\
    &=0
  \end{split}
\end{equation}
where we have used in the last inequality that by \eqref{eq:54} and \eqref{eq:64},
\[
  t\L^1(I\setminus I_j)-\beta_j\leq 
 \int_{Q\setminus A_j}|\nabla v_j|\d x\leq t+\beta_j\,,
    \]
    which yields $t=\lim_{j\to\infty}\int_{Q\setminus A_j}|\nabla v_j|\d x$. The estimate \eqref{eq:65} implies that
    \[
      \limsup_{j\to\infty}\int_{A_j} G(v_j,\nabla v_j)\d x=0\,.
    \]
    We now turn our attention to the contribution of the good sets
    \[
      E_j:=Q\setminus A_j\,.
    \]
    By Lemma \ref{lem:hbounds}, we have that on $E_j$,
      \begin{equation}\label{eq:66}
        \begin{split}
        \Big|G(v_j,\nabla v_j)&-G\left((a_1,(v_j)_2)^T, \nabla v_j\right)\Big|\\
        &\leq C\Delta \max \left(G(v_j,\nabla v_j),G\left((a_1,(v_j)_2)^T, \nabla v_j\right)\right)\,.
      \end{split}
    \end{equation}
    By the subadditivity of $\xi\mapsto G(v,\xi)$, $G(v,\xi)\leq C|\xi|$, $\nabla (v_j)_1\to 0$ in $L^1$ and $\partial_{x_1} (v_j)_2\to 0$ in $L^1$ (see \eqref{eq:27}), we have that
    \begin{equation}
      \label{eq:67}
      \begin{split}
      \limsup_{j\to\infty} &\int_{E_j} 
      G\left((a_1,(v_j)_2)^T, \nabla v_j\right)\d x\\
     & \leq \limsup_{j\to\infty} \int_{E_j} G\left((a_1,(v_j)_2)^T, \left(\begin{array}{cc}0& 0\\0&\partial_{x_2} (v_j)_2\end{array}\right)\right)\d x\,.
   \end{split}
    \end{equation}
    Combining \eqref{eq:66} and \eqref{eq:67}, we obtain
    \[
      \begin{split}
\limsup_{j\to\infty} &\int_{E_j} 
G\left(v_j, \nabla v_j\right)\d x\\
&\leq
\frac{1+C\Delta}{1-C\Delta}
\limsup_{j\to\infty} \int_{E_j} G\left((a_1,(v_j)_2)^T, \left(\begin{array}{cc}0& 0\\0&\partial_{x_2} (v_j)_2\end{array}\right)\right)\d x\,.
\end{split}
\]
By Lemma \ref{lem:G1d_estim}, we obtain
\[
  \begin{split}
\frac{1-C\Delta}{1+C\Delta}  \limsup_{j\to\infty}& \int_{E_j} 
  G\left(v_j, \nabla v_j\right)\d x\\
  &\leq
2\sqrt{1+a_1^2}\int_{-1/2}^{1/2} \arccos \NN(v_j(x_1,-1/2))\cdot\NN(v_j(x_1,1/2))\d x_1\\
  &\quad+4\sqrt{1+a_1^2}\int_Q\left|\min(0,\partial_{x_2}(v_j)_2)\right|\d x\,.
  \end{split}
  \]
Since $v_j(\cdot,-1/2)\to a$ and $v_j(\cdot,1/2)\to b$ in $L^1(I)$, and 
\[
  \int_Q\left|\min(0,\partial_{x_2}(v_j)_2)\right|\d x\to 0\,,
  \]
  the proof of the proposition is completed by sending $\Delta\to 0$.
\end{proof}

\appendix

\section{Proof of Proposition \ref{prop:Q2F}}
\label{sec:proof-prop-refpr}
\begin{proof}[Proof of Proposition \ref{prop:Q2F}]
  First we prove $\sqrt{1+|v|^2}g_\lambda(S(v,\cdot))\leq Q_2 f_\lambda(v,\cdot)$. Indeed, we prove
  the slightly stronger claim $\sqrt{1+|v|^2}g_\lambda(S(v,\cdot))\leq Q_1 f(v,\cdot)$,
  following the proof of Theorem 6.28 in \cite{MR2361288}, where this is proved
  for $\lambda=1$ and $v=0$. The modifications that are necessary with respect
  to that proof are minor, so we will be brief. 

First one shows that $g_1(S(v,\cdot))$ is polyconvex
  by defining
\[
\theta(t)=\begin{cases}2t &\text{ if }t\leq 1\\1+t^2 &\text{ else,}\end{cases}
\]
and convex functions $H_\pm:\R^{2\times 2}\times \R\to\R$ by
\[
H_\pm(\xi,A)=\theta\left(\left(|S(v,\xi)|^2\pm 2\det
    S(v,\xi)\right)^{1/2}\right)\mp \frac{2A}{(1+|v|^2)^2}\,.
\]
Defining the convex function $H(\xi,A)=\max (H_+(\xi,A),H_-(\xi,A))$,  one can then show
that $g_1(S(v,\xi))=H(\xi,\det\xi)$ and hence $g_1(S(v,\cdot))$ is polyconvex. Furthermore,
we have that $g_\lambda(S(v,\cdot))=\sqrt{\lambda} g_1(S(v,\cdot/\sqrt{\lambda}))$ and
hence $g_\lambda(S(v,\cdot))$ is polyconvex for every $\lambda>0$. The
inequality $\sqrt{1+|v|^2}g_\lambda(S(v,\cdot))\leq f_\lambda(v,\cdot)$ can be verified from the definitions. This
proves $\sqrt{1+|v|^2} g_\lambda(S(v,\cdot))\leq Q_2 f_\lambda(v,\cdot)$.
It
remains to show the opposite inequality. To do so, we make the following
definition:
\begin{definition}
  Let $f:\Rnsym\to \R$. We say that $f$ is symmetric rank one convex if
\[
f(t\xi_1+(1-t)\xi_2)\leq t f(\xi_1)+(1-t) f(\xi_2) 
\]
for all $t\in[0,1]$, and for all $\xi_1,\xi_2\in\Rnsym$ such that $\xi_1-\xi_2=\alpha\eta\otimes \eta$ for
some $\alpha\in \R$, $\eta\in\R^n$.

Furthermore, for $f:\Rnsym\to \R$, we set
\[
\Rs f(\xi):=\sup\{g(\xi):g\leq f \text{ and } g \text{ is symmetric rank one
  convex}\}\,.
\]
\end{definition}
By Lemma B.5 of \cite{olbermann2017michell}, we have $Q_2f_\lambda(v,\cdot)\leq
\Rs f_\lambda(v,\cdot)$. In Lemma \ref{lem:RsGl} below, the inequality $\Rs
f_\lambda(v,\cdot)\leq h_\lambda(v,\cdot)$ is proved. This completes the proof
of the proposition.
\end{proof}

\begin{lemma}
\label{lem:RsGl}
  We have
\[
\Rs f_\lambda(v,\cdot)\leq h_\lambda(v,\cdot)\,.
\]
\end{lemma}
\begin{proof}
We write $\bar f_\lambda:=(1+|v|^2)^{-1/2}f_\lambda(v,\cdot)$, and need to show
$\Rs \bar f_\lambda\leq g_\lambda(S(v,\cdot))$.

Let $\xi\in \Rsym$. Then there exists an orthonormal basis $\tilde e_1,\tilde
e_2$ and $x,y\in \R$ such that
\[
S(v,\xi)=x\tilde e_1\otimes \tilde e_1+y \tilde e_2\otimes \tilde e_2\,.
\]
We may assume $|x|+|y|< \sqrt{\lambda}$, since otherwise we know
$\bar f_\lambda(\xi)=g_\lambda(S(v,\xi))=\sqrt{\lambda}+\frac{x^2+y^2}{\sqrt{\lambda}}$. Similarly, we may assume
$0<|x|+|y|$, since otherwise $\bar f_\lambda(\xi)=g_\lambda(S(v,\xi))=0$.
Let
$\alpha,\beta\in (0,1)$ to be chosen later, and set
\[
S_1=\left(\begin{array}{cc} 0&0\\0&0\end{array}\right)\,,\quad
S_2=\left(\begin{array}{cc} x/\alpha &0\\0&0\end{array}\right)\,,\quad
S_3=\left(\begin{array}{cc} x&0\\0&y/\beta\end{array}\right)\,,
\]
where all matrices are in the $\tilde e_1,\tilde e_2$ basis. Observing that  the map
$\sigma\mapsto S(v,\sigma)$  is linear and invertible, we may
choose $\xi_i$ such that $S(v,\xi_i)=S_i$ for $i=1,2,3$. 

Note that $\beta S_3+(1-\beta)(\alpha S_2+(1-\alpha) S_1))=S(v,\xi)$, and
$S_3-(\alpha S_2+(1-\alpha)S_1),S_2-S_1$ are both symmetric-rank-one. By
linearity, we also have that $\beta \xi_3+(1-\beta)(\alpha \xi_2+(1-\alpha)
\xi_1))=\xi$ and that $\xi_3-(\alpha\xi_2+(1-\alpha)\xi_1),\xi_2-\xi_1$ are
symmetric-rank-one. (These relations provide a ``laminate'' of order two.)
By Lemma B.7 in \cite{olbermann2017michell} (which states that $\Rs \bar
f_\lambda$ may be estimated from above by laminates of order $k$ for any $k\in\N$), 
we have
\[
\begin{split}
  \Rs \bar f_\lambda(\xi)\leq &\beta  \bar f_\lambda(\xi_3)+(1-\beta)\left(\alpha
    \bar f_\lambda(\xi_2)+(1-\alpha) \bar f_\lambda(\xi_1)\right)\\
  =&
  \beta\left(\sqrt{\lambda}+\frac{x^2+y^2/\beta^2}{\sqrt{\lambda}}\right)+(1-\beta)\alpha\left(\sqrt{\lambda}+\frac{x^2/\alpha^2}{\sqrt{\lambda}}\right)\,.
\end{split}
\]
Now we assume $|x|>0$. The right hand side in the last estimate is convex in $\alpha$; it attains its minimum at $\alpha=
\frac{|x|}{\sqrt{\lambda}}$. Hence, 
\[
\begin{split}
  \Rs \bar f_\lambda(\xi)\leq &\beta
  \left(\sqrt{\lambda}+\frac{x^2+y^2/\beta^2}{\sqrt{\lambda}}\right)+(1-\beta)2|x|\\
  =& 2|x|+\frac{\beta}{\sqrt{\lambda}}\left(\sqrt{\lambda}-\frac{|x|}{\sqrt{\lambda}}\right)^2+\frac{y^2/\beta}{\sqrt{\lambda}}
\end{split}
\]
Choosing $\beta=|y|/(\sqrt{\lambda}-|x|)$, we obtain
\[
\Rs  \bar f_\lambda(\xi)\leq 2\left(|x|+|y|-\frac{|xy|}{\sqrt{\lambda}}\right)=g_\lambda(S(v,\xi))\,.
\]
It remains to prove the claim for the case $|x|=0$. Then we have 
\[
\begin{split}
  \Rs  \bar f_\lambda(\xi)\leq &\beta  \bar f_\lambda(\xi_3)+(1-\beta) \bar f_\lambda(\xi_1)\\
  =&
  \beta\left(\sqrt{\lambda}+\frac{x^2+y^2/\beta^2}{\sqrt{\lambda}}\right)\,.
\end{split}
\]
Again setting $\beta=|y|/(\sqrt{\lambda}-|x|)$, we obtain the same conclusion as before.
This proves the lemma.
\end{proof}

\bibliographystyle{alpha}
\bibliography{graphs}

\newcommand{\etalchar}[1]{$^{#1}$}
\begin{thebibliography}{VSOdlC11}

\bibitem[ADM92]{ambrosio1992relaxation}
L.~Ambrosio and G.~Dal~Maso.
\newblock On the relaxation in {$BV (\Omega;\mathbb{R}^m)$} of quasi-convex
  integrals.
\newblock {\em J. Funct. Anal.}, 109(1):76--97, 1992.

\bibitem[AF84]{acerbi1984semicontinuity}
E.~Acerbi and N.~Fusco.
\newblock Semicontinuity problems in the calculus of variations.
\newblock {\em Archive for Rational Mechanics and Analysis}, 86(2):125--145,
  1984.

\bibitem[AFP00]{MR1857292}
L.~Ambrosio, N.~Fusco, and D.~Pallara.
\newblock {\em Functions of bounded variation and free discontinuity problems}.
\newblock Oxford Mathematical Monographs. The Clarendon Press, Oxford
  University Press, New York, 2000.

\bibitem[AK93]{allaire1993optimal}
G.~Allaire and R.~V. Kohn.
\newblock Optimal design for minimum weight and compliance in plane stress
  using extremal microstructures.
\newblock {\em European journal of mechanics. A. Solids}, 12(6):839--878, 1993.

\bibitem[Alb93]{alberti1993rank}
G.~Alberti.
\newblock Rank one property for derivatives of functions with bounded
  variation.
\newblock {\em Proceedings of the Royal Society of Edinburgh Section A:
  Mathematics}, 123(2):239--274, 1993.

\bibitem[Bon26]{bonnesen1926quelques}
T.~Bonnesen.
\newblock Quelques probl{\`e}mes isop{\'e}rim{\'e}triques.
\newblock {\em Acta Mathematica}, 48(1-2):123--178, 1926.

\bibitem[CK62]{caspar1962physical}
D.~L.~D. Caspar and Aaron Klug.
\newblock Physical principles in the construction of regular viruses.
\newblock In {\em Cold Spring Harbor Symposia on Quantitative Biology},
  volume~27, pages 1--24. Cold Spring Harbor Laboratory Press, 1962.

\bibitem[Dac08]{MR2361288}
B.~Dacorogna.
\newblock {\em Direct methods in the calculus of variations}, volume~78 of {\em
  Applied Mathematical Sciences}.
\newblock Springer, New York, second edition, 2008.

\bibitem[Del91]{delladio1991lower}
S.~Delladio.
\newblock Lower semicontinuity and continuity of functions of measures with
  respect to the strict convergence.
\newblock {\em Proceedings of the Royal Society of Edinburgh Section A:
  Mathematics}, 119(3-4):265--278, 1991.

\bibitem[Dem84]{demengel1984fonctions}
F.~Demengel.
\newblock Fonctions {\`a} hessien born{\'e}.
\newblock {\em Ann. I. Fourier}, 34(2):155--190, 1984.

\bibitem[Dem89]{demengel1989compactness}
F.~Demengel.
\newblock Compactness theorems for spaces of functions with bounded derivatives
  and applications to limit analysis problems in plasticity.
\newblock {\em Arch. Ration. Mech. Anal.}, 105(2):123--161, 1989.

\bibitem[FLP03]{fonseca2003lower}
I.~Fonseca, G.~Leoni, and R.~Paroni.
\newblock On lower semicontinuity in bh and 2-quasiconvexification.
\newblock {\em Calculus of Variations and Partial Differential Equations},
  17(3):283--309, 2003.

\bibitem[FM93]{MR1218685}
I.~Fonseca and S.~M{\"u}ller.
\newblock Relaxation of quasiconvex functionals in {${\mathrm
  BV}(\Omega,{\mathbf R}^p)$} for integrands {$f(x,u,\nabla u)$}.
\newblock {\em Arch. Ration. Mech. Anal.}, 123(1):1--49, 1993.

\bibitem[FMP98]{fonseca1998analysis}
I.~Fonseca, S.~M{\"u}ller, and P.~Pedregal.
\newblock Analysis of concentration and oscillation effects generated by
  gradients.
\newblock {\em SIAM journal on mathematical analysis}, 29(3):736--756, 1998.

\bibitem[Hut86]{hutchinson1986second}
J.~E. Hutchinson.
\newblock Second fundamental form for varifolds and the existence of surfaces
  minimising curvature.
\newblock {\em Indiana University Mathematics Journal}, 35(1):45--71, 1986.

\bibitem[KS86]{MR820342}
R.~V. Kohn and G.~Strang.
\newblock Optimal design and relaxation of variational problems. {I}, {II} and
  {III}.
\newblock {\em Comm. Pure Appl. Math.}, 39:113--137, 139--182,353--377, 1986.

\bibitem[LMN03]{lidmar2003virus}
J.~Lidmar, L.~Mirny, and D.~R. Nelson.
\newblock Virus shapes and buckling transitions in spherical shells.
\newblock {\em Physical Review E}, 68(5):051910, 2003.

\bibitem[MA99]{macgillivray1999structural}
L.~R. MacGillivray and J.~L. Atwood.
\newblock Structural classification and general principles for the design of
  spherical molecular hosts.
\newblock {\em Angewandte Chemie International Edition}, 38(8):1018--1033,
  1999.

\bibitem[Man96]{mantegazza1998curvature}
C.~Mantegazza.
\newblock Curvature varifolds with boundary.
\newblock {\em J. Diff. Geom.}, 43:807--843, 1996.

\bibitem[Mar85]{marcellini1985approximation}
P.~Marcellini.
\newblock Approximation of quasiconvex functions, and lower semicontinuity of
  multiple integrals.
\newblock {\em manuscripta mathematica}, 51(1-3):1--28, 1985.

\bibitem[Mey65]{MR0188838}
N.~G. Meyers.
\newblock Quasi-convexity and lower semi-continuity of multiple variational
  integrals of any order.
\newblock {\em Trans. Amer. Math. Soc.}, 119:125--149, 1965.

\bibitem[Min89]{minkowski1989volumen}
H.~Minkowski.
\newblock Volumen und {O}berfl{\"a}che.
\newblock In {\em Ausgew{\"a}hlte Arbeiten zur Zahlentheorie und zur
  Geometrie}, pages 146--192. Springer, 1989.

\bibitem[Olb17]{olbermann2017michell}
H.~Olbermann.
\newblock Michell trusses in two dimensions as a {$\Gamma$}-limit of optimal
  design problems in linear elasticity.
\newblock {\em Calculus of Variations and Partial Differential Equations},
  56(6):166, 2017.

\bibitem[SOdlC12]{PhysRevE.85.050501}
R.~Sknepnek and M.~Olvera de~la Cruz.
\newblock Nonlinear elastic model for faceting of vesicles with soft grain
  boundaries.
\newblock {\em Phys. Rev. E}, 85:050501, May 2012.

\bibitem[ST98]{savare1998superposition}
G.~Savar{\'e} and F.~Tomarelli.
\newblock Superposition and chain rule for bounded hessian functions.
\newblock {\em Advances in mathematics}, 140(2):237--281, 1998.

\bibitem[VSOdlC11]{PNAS}
G.~Vernizzi, R.~Sknepnek, and M.~Olvera de~la Cruz.
\newblock Platonic and archimedean geometries in multicomponent elastic
  membranes.
\newblock {\em PNAS}, 108:4292--4296, 2011.

\bibitem[YKH{\etalchar{+}}08]{yeates2008protein}
T.~O. Yeates, C.~A. Kerfeld, S.~Heinhorst, G.~C. Cannon, and J.~M. Shively.
\newblock Protein-based organelles in bacteria: carboxysomes and related
  microcompartments.
\newblock {\em Nature Reviews Microbiology}, 6(9):681--691, 2008.

\end{thebibliography}

\end{document}